\documentclass[11pt, a4paper]{amsart}

\usepackage{amsmath}
\newtheorem{theorem}{Theorem}[section]
\newtheorem{definition}{Definition}[section]
\newtheorem{lemma}{Lemma}[section]
\newtheorem{proposition}{Proposition}[section]
\newtheorem{corollary}{Corollary}[section]

\usepackage[colorlinks=true,citecolor=cyan,backref=page]{hyperref}

\DeclareMathOperator{\N}{\mathbb N}

\DeclareMathOperator{\id}{id}
\DeclareMathOperator{\MOD}{MOD}

\usepackage{amsmath}
\usepackage{amsfonts}
\usepackage{amssymb}
\usepackage{stmaryrd} 
\usepackage{amssymb} 
\usepackage{mathrsfs} 
\usepackage{amsfonts} 
\usepackage{esint} 
\usepackage{array} 
\usepackage{latexsym}
\usepackage{lmodern} 
\usepackage{graphicx} 
\usepackage{graphics} 
\usepackage{epsfig} 
\usepackage{color} 
\usepackage[dvipsnames]{xcolor}
\usepackage[all]{xy}
\usepackage{textcomp}
\usepackage{stmaryrd}
\DeclareGraphicsExtensions{.eps,.pdf,.jpeg,.png} 
\usepackage{fancyhdr}
\usepackage{multirow} 
\usepackage{comment}
\usepackage{pgf,tikz}
\usepackage{paralist}
\usepackage{faktor}
\usepackage{hyperref}
\usepackage{url}
\usepackage{hyperref}
\usepackage{blkarray,bigstrut}
\usepackage{enumerate}

\title{An order on  circular permutations}

\author{Antoine Abram, Nathan Chapelier-Laget, Christophe Reutenauer}
\address{
D\'epartement de math\'ematiques, Universit\'e du Qu\'ebec \`a Montr\'eal}

\email{abram.antoine@courrier.uqam.ca}
\email{nathan.chapelier@gmail.com}
\email{Reutenauer.Christophe@uqam.ca}

\begin{document}

\maketitle

\tableofcontents

\begin{abstract} Motivation coming from the study of affine Weyl groups, a structure of ranked poset is defined on the set of circular permutations in $S_n$ (that is, $n$-cycles). It is isomorphic to the poset of so-called admitted vectors, and to an interval in the affine symmetric group $\tilde S_n$ with the weak order. The poset is a semidistributive lattice, and the rank function, whose range is cubic in $n$, is computed by some special formula involving inversions.
We prove also some links with Eulerian numbers, triangulations of an $n$-gon, and Young's lattice.
\end{abstract}

\section{Introduction}

In \cite{C}, the second-named author defines for each affine Weyl group an affine variety whose integral points are in bijection with the group. The irreducible components of the variety are affine subspaces and are in bijection with the alcoves in a certain polytope; this set of alcoves is partially ordered: the covering relation of this order is determined by applying a reflection sending an alcove to a neighbouring one, further from the origin.

In the case of the affine symmetric group $\tilde S_n$, it turns out that  the irreducible components are naturally in 
bijection with the set circular permutations in $S_n$ (that is, $n$-cycles), so that one obtains an order on this set. It is 
this poset, with its three instances, that we study in the present article.

The article is self-contained and may be read independently of the motivating article \cite{C}. We begin by defining a partial order on the set of circular permutations in $S_n$, obtaining a ranked poset (Corollary \ref{order1}), whose rank function has image the interval $\{0,1,\ldots,\binom{n}{3}\}$. Note the rather unusual maximum rank, which is {\em cubic} in $n$. The rank function is computed using inversions in a rather subtle way: it is a signed count of inversions (see Eq.(\ref{N})), and it is clearly not the usual length function in the symmetric group). 

The edges of the Hasse diagram of this order are indexed by transpositions $(i,j)$ with $i+1<j$; the covering relation conjugates two circular permutations by this transposition, under the condition that the smaller permutation sends $j$ onto $i$ (we call this large circular descent, see Definition \ref{large circular permutation}). The smallest element in the poset is $(1,2,\ldots,n)
$ and the largest one is its inverse. Inversion and conjugation by the longest element in $S_n$ are anti-automorphisms of the poset. See Figures \ref{P4} and \ref{P5}.

{\em Admitted vectors}, introduced by the second-named author in \cite{C}, are vectors $v$ of natural integers, indexed by the transpositions considered above, which satisfy the conditions
$$
\left\{
\begin{array}{ll}
v_{ij}+v_{jk} \leq v_{ik} \leq v_{ij}+v_{jk}+1~~\text{~~for all}~i<j<k, \\
v_{i,i+1}=0~~\text{~~for all}~1\leq i<n.
\end{array}
\right.
$$

Admitted vectors are naturally ordered, and we show that the poset of admitted vectors is isomorphic with the previous poset of circular permutations (Theorem \ref{iso}). 

Using the isomorphism of posets, we derive several properties. The poset is a lattice, and we give an algorithmic 
construction for the supremum and the infimum of two elements, using the vector incarnation of the poset (Theorem \ref{lattice}). The number 
of edges in the Hasse diagram are counted by Eulerian numbers: precisely, the number of circular permutations in 
$S_{n+1}$ which is covered by $k$ circular permutations is the Eulerian number $a(n,k)$ (= number of permutations in 
$S_n$ having $k$ descents) (Theorem \ref{coverk}); in particular, the number of join-irreducible elements is $2^n-n-1$; the poset is however not isomorphic to $S_n$ with the weak order: indeed the maximum rank is cubic in $n$, whereas the maximum rank for the weak order is  quadratic. We consider also the limit poset when $n\to\infty$: that is, we show that for any $k$, the 
posets coincide for $n$ large enough in their elements of rank $\leq k$; and the limit of the posets is Young's lattice of 
partitions (Theorem \ref{limit}). 

We give a link with triangulations of an $n$-gon. We consider the functions $\delta_{ijk}$, derived from the above defining identities of admitted vectors, that is,
$$
\delta_{ijk}(v)=v_{ik}-v_{ij}-v_{jk}, i<j<k.
$$
These functions take their values in $\{0,1\}$, and satisfy the Ptolemy-like relation
$$
\delta_{ijk}+\delta_{ikl}=\delta_{ijl}+\delta_{jkl}, i<j<k<l.
$$
We show that, given any triangulation of the $n$-gon, the component indexed by $(1,n)$ of any admitted vector $v$ is equal to the sum of all $\delta_{ijk}(v)$, where the sum is over all triangles $i,j,k$ of the triangulation (Theorem \ref{triang}).

The third incarnation of the poset, close to the initial motivation \cite{C}, is an interval in the affine symmetric group $\tilde S_n$. We construct a special element $f_c$ in this group, and show that our poset is isomorphic with the interval $[\id,f_c]$ with the left weak order (Theorem \ref{interval}). The rank function is therefore the length function in this Coxeter group. 

From this isomorphism we derive several consequences: by a result of Reading and Speyer, the lattice is semidistributive (but not distributive, nor modular) (Corollary \ref{semi-distrib}); and by a result of Bj\"orner and Brenti, the M\"obius function takes values in $\{-1,0,1\}$ (see end of Subsection \ref{--interval}).

Finally, in Section \ref{lines}, we illustrate the bijection between circular permutations and admitted vectors using lines diagrams.

\section{Ordering circular permutations}\label{circular}

A conjugate of word $w$ is a word $w'$ such that for some words $x,y$, one has $w=xy,w'=yx$. 
A {\em factor} of a word $w$ is a word $u$ such that $w=xuy$ for some words $x,y$; a {\em circular factor} of $w$ is a factor of some conjugate of $w$. We say that the word $sr$ of length 2 is a {\em subword} of $w$ if $w=xsyrz$ for some words $x,y,z$.

For any permutation $w\in S_n$, that we view as a word on the alphabet $\{1,\ldots,n\}$, and for any $i,j$ with $1\leq i<j\leq n$, define the function $\gamma_{ij}(w)=1$ if $j$ appears before $i$ in $w$, and 0 otherwise; that is, $\gamma_{ij}(w)=1$ if $ji$ is an {\em inversion by value} in $w$ (which means that $ji$ is a subword of the word $w$), and $=0$ if not. Then we define
\begin{equation}\label{N}
N(w)=\sum_{1\leq k\leq n-1}k(n-k)\gamma_{k,k+1}(w)-\sum_{1\leq i<j\leq n}\gamma_{ij}(w).
\end{equation}

\begin{proposition}\label{Nprop} (i) If $w,w'\in S_n$ are conjugate words, then $N(w)=N(w')$.

(ii) Suppose that $w$ has the circular factor $sr$, $s,r\in\{1,2,\ldots,n\}$, with $s>r+1$ and let $w'$ be obtained from $w$ by exchanging $r$ and $s$. Then $N(w')=N(w)+1$.

(iii) $N(12\ldots n)=0$.

(iv) $N(n\ldots 21)=\binom{n}{3}$. 

\end{proposition}

\begin{proof} (i) It suffices to prove this in the case where $w=ru$, $w'=ur$, $r\in \{1,\ldots,n\}$. Then by inspection of the inversions of $w$ and $w'$, we see that: for any $i<r$, $w$ has the inversion $ri$, but $w'$ not; for any $j>r$, $w'$ has the inversion $jr$, but $w$ not; all other inversion are identical in $w$ and $w'$. Finally we synthesize this for all $1 \leq i < r < j \leq n$ and $k,l \neq r$ by:
$$
 \left\{
\begin{array}{rl}
  \gamma_{rj}(w) =0 ~~\text{and} ~~\gamma_{ir}(w)=1\\
  \gamma_{rj}(w') =1 ~~\text{and}~~ \gamma_{ir}(w')=0\\
\gamma_{kl}(w) = \gamma_{kl}(w').
\end{array}
\right. \\
$$

Thus it follows that
\begin{align*}
N(w)-N(w')  = & ~(r-1)(n-(r-1))\gamma_{r-1,r}(w) - \sum\limits_{i=1}^{r-1}\gamma_{i,r}(w) ~ - \\
&  ~r(n-r)\gamma_{r,r+1}(w')+ \sum\limits_{j=r+1}^n\gamma_{rj}(w')  \\
= & ~r(n-r)-(r-1)(n-r+1)-(n-r)+(r-1) \\
 = & ~rn-r^2-rn+r^2-r+n-r+1-n+r+r-1 \\
 = &~ 0.
\end{align*}

(ii) By (i), we may assume that $sr$ is a factor of $w$. Since $s>r+1$, $k(k+1)$ (resp. $(k+1)k)$ is simultaneously a subword of $w$ and $w'$, or not; moreover, there is one subword of the form $ji$, $j>i$, less in $w'$ than in $w$; hence $N(w')=N(w)+1$.

(iii) This is straightforward.

(iv) Recall that the sum of the integers from 1 to $n-1$ (resp. of their squares) is equal to $(n-1)n/2$ (resp. $(n-1)n(2n-1)/6$). Since ${\gamma_{ij}(n...21) = 1}$ for all $i <j$, it follows that 
\begin{align*}
N(n...21) & = \sum\limits_{k=1}^{n-1}k(n-k) - \sum\limits_{1 \leq i <j \leq n}1  =  n\sum\limits_{k=1}^{n-1}k  -  \sum\limits_{k=1}^{n-1}k^2 - \frac{n(n-1)}{2} \\ 
                & = \frac{n^2(n-1)}{2}  - \frac{n(n-1)(2n-1)}{6} - \frac{n(n-1)}{2} \\
                & = \frac{n(n-1)(n-2)}{6}.
\end{align*}
\end{proof}

In this article we call {\em circular permutation} an $n$-cycle in $S_n$. If $w\in S_n$, viewed as word, denote by $(w)$ the circular permutation in $S_n$, written in cyclic form. Note that in the article, {\em we use the cycle form notation either as usually as a sequence of digits separated by commas, or as word without putting commas}, and in both cases, surrounded by parenthesis; for example, $(12345)$ and $(1,2,3,4,5)$ both represent the canonical $5$-cycle. It is well-known that a circular permutation in $S_n$ has exactly $n$ such representations, which are the $(w')$, $w'$ conjugate of $w$. In particular, there is exactly one representation beginning by 1, that is $(w')$ with $w'=1v$.

\begin{corollary} The function $N$ induces a function on the set of circular permutations in $S_n$, that we denote also by $N$. We synthesize this with the following commutative diagram:
\end{corollary}

\begin{center}
\begin{tikzpicture} 
\node at (0, 0) {$\begin{small}\left\{\begin{array}{c}
    \text{Words associated to} \\
    \text{permutations of }S_n
    \end{array} \right\}\end{small}$} ;

\node at (0, -3) {$\quad ^{\begin{small}
\left\{\begin{array}{c}
    \text{Words associated to} \\
    \text{permutations of }S_n
    \end{array} \right\}
    \end{small}}\Big/_{\text{Conjugation}}$} ;

\node at (0, -6.5){$\begin{small}\left\{\begin{array}{c}
    \text{Circular} \\
    \text{ permutations of }S_n
    \end{array} \right\}\end{small}$};
    
  \node at (6.5,0) {$\mathbb{Z}.$};
  
  \node at (4, 0.3) {$N$};
  
  \node at (-0.3, -1.2) {$\pi$};
  
  \node at (4.6,-4){$N$};

\node at (3, -1.2){$\circlearrowleft$};

  \node at (2, -5){$\circlearrowleft$};
  
  \node at (-0.2, -4.8){ \rotatebox{90}{$\simeq$}};

\draw [->] (0,-0.6) -- (0,-1.9) ; 
\draw [->] (0,-3.6) -- (0,-5.7) ;
\draw [->] (2.3,0) -- (6,0) ;
\draw  [->] (2.8,-2.8) -- (6.2,-0.35) ;
\draw[dashed]  [->] (2.3,-6.5) -- (6.3,-0.6) ;
\end{tikzpicture}
\end{center}

\newpage

Define a relation of circular permutation by $(w)\to (w')$ if $w$ has a circular factor $sr$ with $s>r+1$ and $w'$ is 
obtained by replacing in $w$ this factor by $rs$. Clearly, this relation does not depend on the chosen representatives of 
the circular permutations. For example, $(12345)\to (52341)$ is seen also by $(34512)\to (34152)$.

\begin{corollary}\label{order1} The reflexive and transitive closure of the relation $\to$ is a partial order on the set of circular permutations. The corresponding poset is graded by $N$, its smallest element is $(12\cdots n)$ and its largest one is $(n\cdots 21)$.
\end{corollary}

\begin{proof} That it is an order follows since there is no cycle in the transitive closure of $\to$, by Proposition \ref{Nprop} (ii). 

Note that if for a circular permutation $(w)$, $w$ has the property of having no circular factor $sr$ with $s>r+1$, then $
(w)=(n\cdots 21)$: indeed we may assume that $w$ begins by $n$, and then the second letter must be $n-1$, the 
third one must be $n-2$, and so on. The element $(n\ldots 21)$ is the maximum of the poset; indeed, it is the only one 
having the previous property; moreover, any element $\sigma$ having some circular factor $sr$ ($s>r+1$) is smaller 
than some other element, by the definition of $\to$, so that by iteration, $\sigma<(n\cdots 21)$.

Similarly, $(12\cdots n)$ is the smallest element of the poset.
\end{proof}

\begin{itemize}
\item[]
\item[]
\end{itemize}

\begin{figure}[h!]
\begin{center}
\includegraphics[scale=0.75]{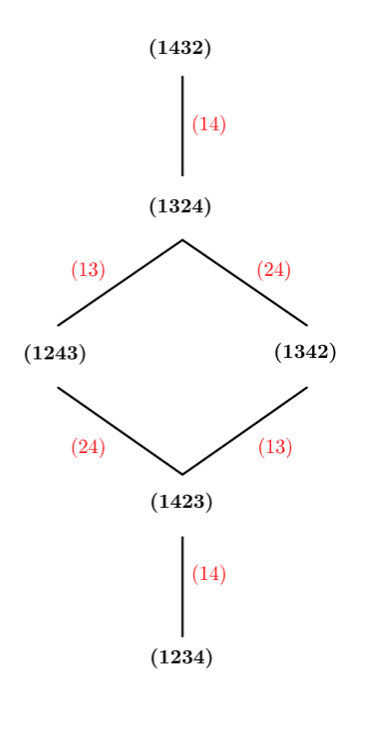} 
\end{center}
\caption{Poset of circular permutations for $n=4$. The red labels indicate the cover relation: it is the corresponding large circular descent of the lower circular permutation.}\label{P4}
\end{figure}

\newpage

\begin{figure}[h!]
\begin{center}
\includegraphics[scale=0.9]{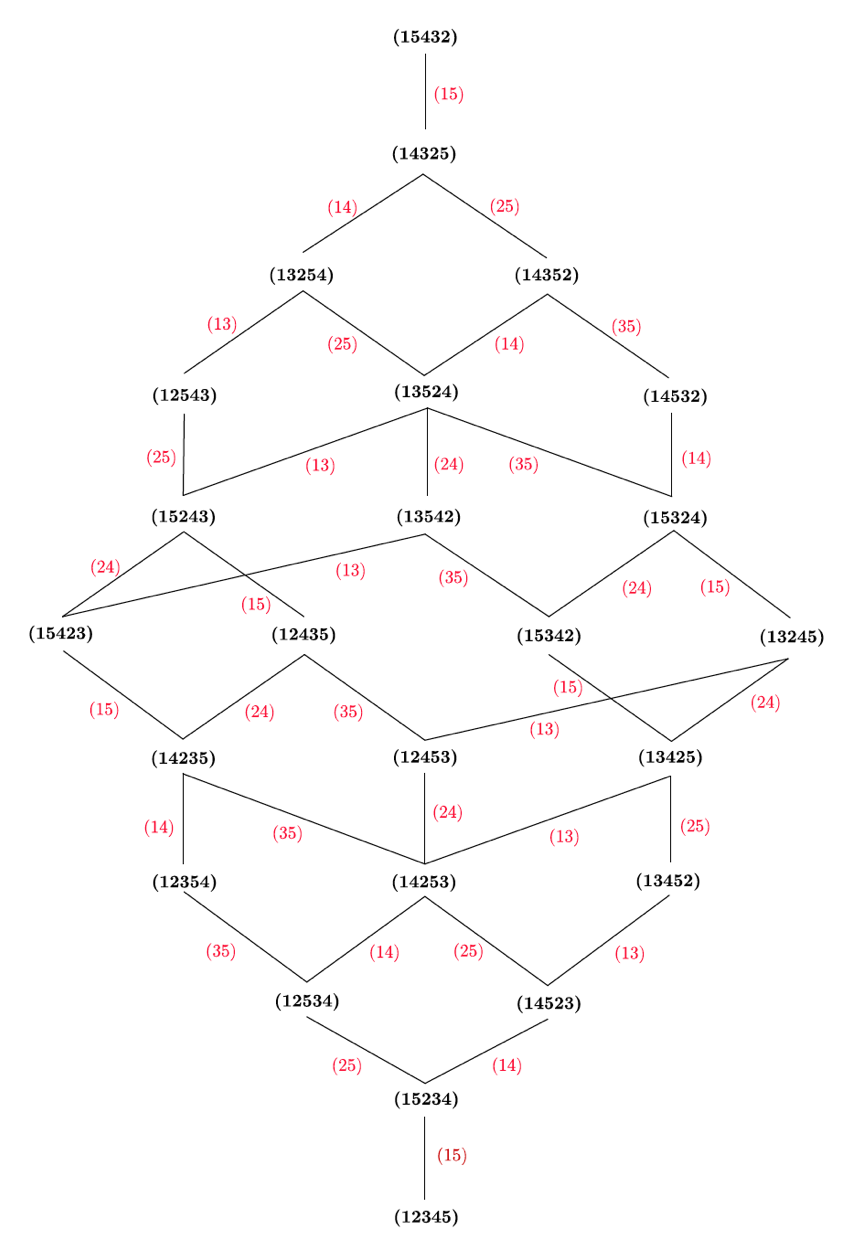} 
\end{center}
\caption{Poset of circular permutations for $n=5$. For the significance of the red labels, see Figure \ref{P4}.}\label{P5}
\end{figure}

\newpage

%\begin{figure}[h!]
%\begin{center}
%\includegraphics[scale=0.47]{Images_posets_A3_CAN.png}
%\end{center}
%\caption{The poset for $n=4$}\label{P4}
%\end{figure}

\begin{corollary} The image of the function  $N$ is $\{0,1,\ldots,\binom{n}{3}\}$.
\end{corollary}

\begin{proof} 
This follows from the previous corollary and from Proposition \ref{Nprop}.
\end{proof}

\begin{corollary}\label{inversion} The functions $\sigma\mapsto \sigma^{-1}$ and $\sigma\mapsto w_0\circ\sigma \circ w_0$ (where ${w_0=n\cdots 21}$ is the longest permutation) are anti-automorphisms of the poset. 
\end{corollary}

In Figures \ref{P4} and \ref{P5}, the first automorphism corresponds to a central symmetry of the Hasse diagram.

\begin{proof} 1. For inversion, this follows from the fact that the inverse of $(w)$ is $(\tilde w)$, where $\tilde w$ is the reversal of $w$. Therefore, cyclic factors $rs$ and $sr$ are interchanged in $(w)$ and its inverse, showing that $(u)\to (v)$ if and only if $(\tilde v)\to (\tilde u)$.

2. Note that a cyclic factor $rs$ gives after applying the second function the cyclic factor $(n+1-r)(n+1-s)$, which implies the result.
\end{proof}

In the Hasse diagram of the poset of circular permutations, each edge is labelled by a transposition $(rs)$, $r+1<s$, 
which is defined by the fact that the largest permutation $\tau$ of this edge has the circular factor $(rs)$, and the 
smallest $\sigma$ has the circular factor $sr$ (see Figures \ref{P4} and \ref{P5}); this transposition clearly conjugates both permutations: $
\tau=(rs)\circ \sigma \circ (rs)$. Taking any upwards path in this diagram, from a circular permutation $\sigma$ to a 
larger one, $\tau$ say, we may perform the corresponding product of transpositions in $S_n$ from right to left, obtaining a 
permutation $\alpha$, which conjugates them: $\tau=\alpha\circ \sigma \circ \alpha^{-1}$. This implies that the class of $\alpha$ modulo the centralizer of $\sigma$ (which has $n$ elements) does not depend on the chosen path. But we may be more precise.

\begin{proposition}\label{path} With the previous notations, $\alpha$ depends only on $\sigma$ and $\tau$ and not on the chosen path. In particular, for a maximal path (thus $\sigma=(12\cdots n)$ and $\tau=(n\cdots 21)=\sigma^{-1}$), one has, as words, $\alpha=n\ldots21$ if $n$ is odd, and $(n/2)\ldots 21 n\ldots (1+n/2)$ if $n$ is even.
\end{proposition}

As examples, for $n=5$ and $6$ the latter permutations are $54321$ and $321654$.

The proof of the proposition will be given at the end of Section \ref{vectors}.

\section{An isomorphism towards admitted vectors}\label{vectors}

Define $S=\{(i,i+1),i=1,\ldots,n-1\}$, $T_1=\{(i,j), 1\leq i, i+1<j\leq n\}$ and $T=S\cup T_1$. 
For a vector $v$ in $\mathbb Z^T$, we write $v_{ij}:=v_{(i,j)}$.
We say that $v\in\N^T$  is {\em admitted} if 
$v_{i,i+1}=0$ for any $i$, and if for any $(i,k)\in T_1$, one has
\begin{equation}\label{admitted}
v_{ij}+v_{jk} \leq v_{ik} \leq v_{ij}+v_{jk}+1~\text{~for all~}~ i<j<k.
\end{equation}

Note that equivalently $v_{ik}-v_{ij}-v_{jk}=0 $ or $1$.

The set of admitted vectors inherits the natural partial order of $\N^T$. We give below the examples of the posets of admitted vectors for $n=4$ and $n=5$.

\newpage

\vspace*{\stretch{1}}
\begin{figure}[h!]
\begin{center}
\includegraphics[scale=0.75]{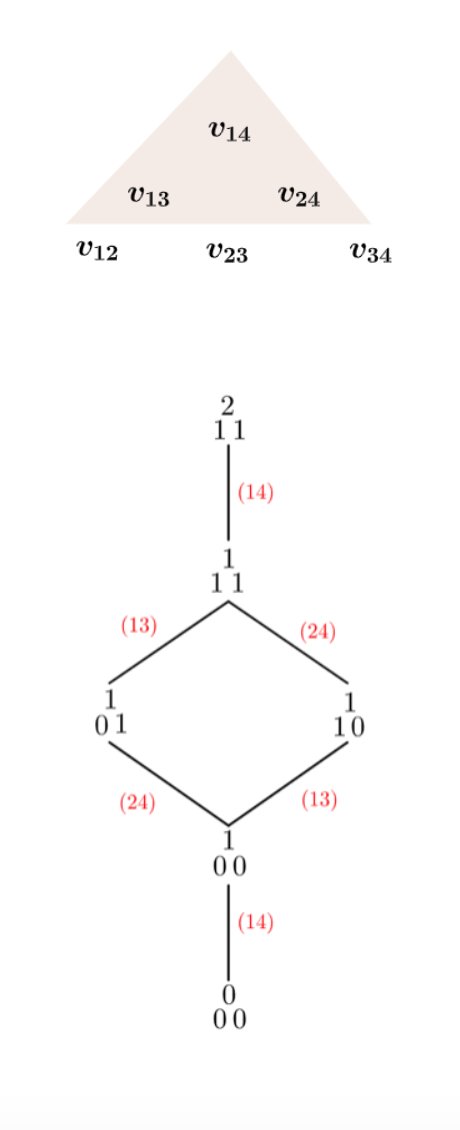} 
\end{center}
\caption{Poset of admitted vectors for $n=4$. The above triangle shows the relative positions of admitted vectors in $\mathbb{N}^6$. In the expression of admitted vectors we drop the first line since the coefficients $v_{i,i+1}=0$. The red labels indicate the cover relation induced by the natural order of $\mathbb{N}^6$. }
\label{PyraA4}
\end{figure}
\vspace*{\stretch{1}}

\newpage

\vspace*{\stretch{1}}
\begin{figure}[h!]
\begin{center}
\includegraphics[scale=0.77]{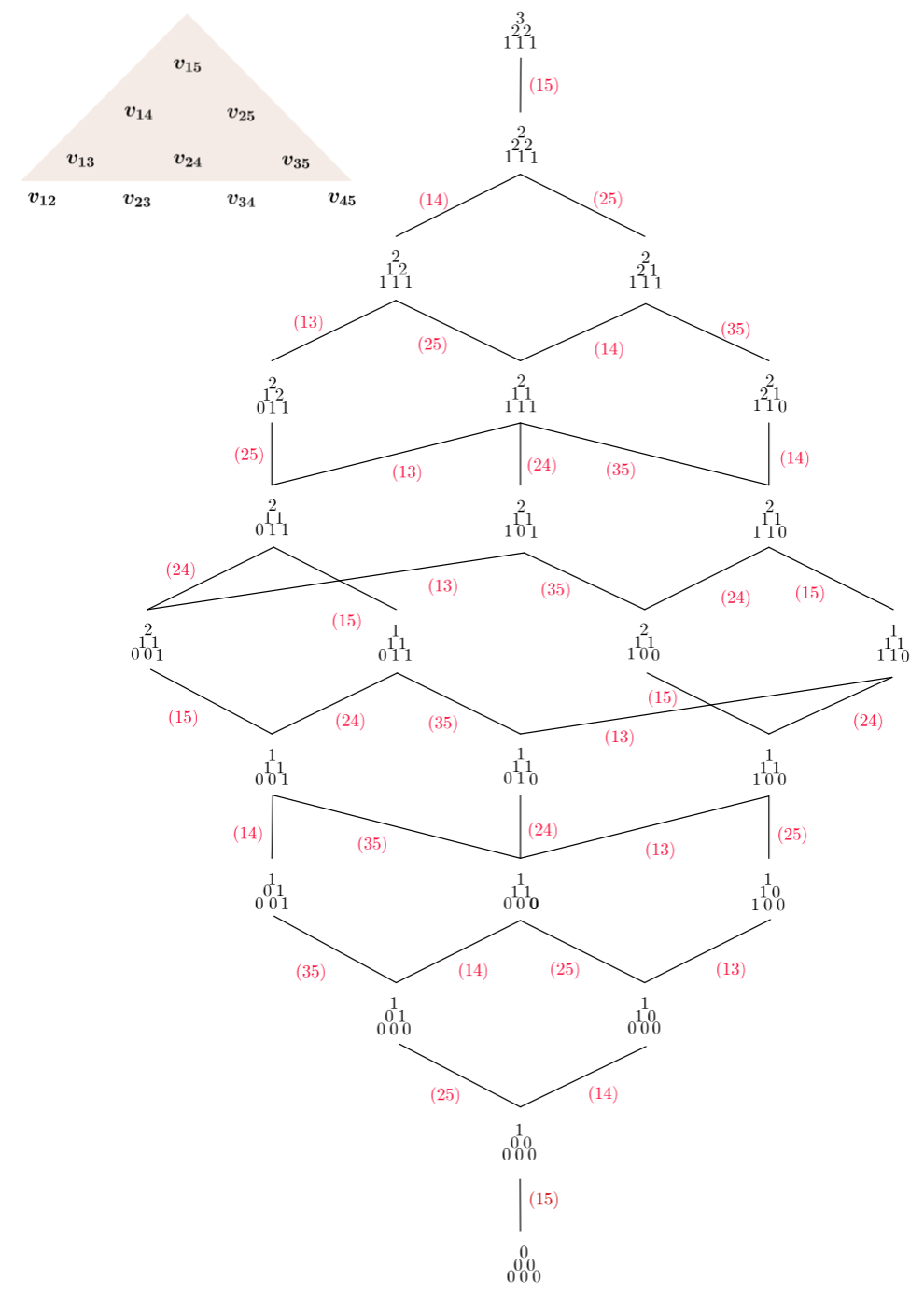} 
\end{center}
\caption{Poset of admitted vectors for $n=5$. See the comments in Figure \ref{PyraA4}.%The above triangle shows the relative positions of admitted vectors in $\mathbb{N}^{10}$. In the expression of admitted vectors we drop the first line since the coefficients $v_{i,i+1}=0$. The red labels indicate the cover relation induced by the natural order of $\mathbb{N}^{10}$. 
}
\label{PyraA5}
\end{figure}
\vspace*{\stretch{1}}

\newpage

We define a mapping $V_0$ from $S_n$ into $\mathbb Z^T$ as follows: let $w\in S_n$ and let $V_0(w)=v\in\mathbb Z^T$ be defined by $$v_{ij}=-\gamma_{ij}(w)+\sum_{i\leq k<j}\gamma_{k,k+1}(w).$$

\begin{proposition}\label{V-0} Each $V_0(w),w\in S_n$, is an admitted vector; $V_0(w)=V_0(w')$ if $w,w'$ are conjugate; and $N(w)$ is equal to the sum of the components of $V_0(w)$.
\end{proposition}

\begin{proof} Let $v=V_0(w)$. Let $1\leq i<j\leq n$. 

If $(i,i+1)\in S$, then by definition of $V_0$, $v_{i,i+1}=-\gamma_{i,i+1}(w)+\gamma_{i,i+1}(w)=0$. 
If $(i,k)\in T$, and $i<j<k$, then 
\begin{align*}
v_{ik}-v_{ij}-v_{jk}  = ~& -\gamma_{ik}(w)+\sum_{i\leq p<k}\gamma_{p,p+1}(w)
                              +\gamma_{ij}(w)-\sum_{i\leq p<j}\gamma_{p,p+1}(w)~+ \\ 
                              &\text{~~} \gamma_{jk}(w)-\sum_{j\leq p<k}\gamma_{p,p+1}(w)\\ 
                              = ~& ~\gamma_{ij}(w)+\gamma_{jk}(w)-\gamma_{ik}(w).
\end{align*}

 By inspecting the relative positions of $i,j,k$ in the word $w$, 
one checks that the latter quantity is equal to 0 or 1. Hence $v$ satisfies the inequalities (\ref{admitted}). Moreover,
$v_{i,i+1}=0$, and by induction on $j-i$, these inequalities imply that the $v_{ij}$ are all nonnegative.

This proves that $v=V_0(w)$ is an admitted vector. We show now that it is invariant under conjugation of $w$.

Suppose that $w=ru,w'=ur$, $r\in \{1,\ldots,n\}$. Let $v=V_0(w),v'=V_0(w')$. Let $i<j$. Then 
$$
v_{ij}=-\gamma_{ij}(w)+\sum_{i\leq k<j}\gamma_{k,k+1}(w) ~\text{~and~}~ v'_{ij}=-\gamma_{ij}(w')+\sum_{i\leq k<j}\gamma_{k,k+1}(w').
$$ 

In the case $j=i+1$, we certainly have $v_{ij}=v'_{ij}$, since they are both equal to 0.

Thus we may assume that $j>i+1$.
Suppose that $i,j\neq r$. If $i<r<j$, there is the subword $r(r-1)$ in $w$, but not in $w'$, and there is the subword $
(r+1)r$ in $w'$, but not in $w$;  thus $v_{ij}=v'_{ij}$. If on the other hand, the double inequality $i<r<j$ does not hold, 
then the two previous sums are identically equal and $v_{ij}=v'_{ij}$, too. 

Suppose that $r=i$. Then $\gamma_{ij}(w)=0$, $\gamma_{ij}(w')=1$, and $\gamma_{i,i+1}(w)=0$, $\gamma_{i,i+1}(w')=1$; the other terms are identical for $w$ and $w'$, therefore $v_{ij}=v'_{ij}$. The case $r=j$ is similar.

We deduce that $V_0(w)=V_0(w')$. Thus $V_0(w)$ is invariant under conjugation of $w$.

Finally, we have 
\begin{align*}
\sum\limits_{i<j}v_{ij} & =\sum\limits_{i<j}(-\gamma_{ij}(w)+\sum\limits_{i\leq k<j}\gamma_{k,k+1}(w)) \\
							    & =-\sum\limits_{i<j}\gamma_{ij}(w)+\sum\limits_{k=1}^n\sum\limits_{i\leq k<j}\gamma_{k,k+1}(w)\\
							    & = -\sum\limits_{i<j}\gamma_{ij}+ \sum\limits_{1\leq k\leq n}k(n-k)\gamma_{k,k+1}(w) \\
							    & =N(w).
\end{align*}

\end{proof}

We introduce the functions $\delta_{ijk}: \N^T\to \N$, $1\leq i\leq j\leq k\leq n$, by $\delta_{ijk}(v)=v_{ik}-v_{ij}-v_{jk}$, where we put $v_{ii}=0$. If $v$ is 
an admitted vector, these functions take by (\ref{admitted}) their value in $\{0,1\}$, and moreover satisfy the following 
relation, for any $i\leq j\leq k\leq l$:
\begin{equation}\label{ptolemee}
\delta_{ijk}(v)+\delta_{ikl}(v)=\delta_{ijl}(v)+\delta_{jkl}(v).
\end{equation}

Indeed, the left-hand side evaluated on the admitted vector $v$ is equal to $v_{ik}-v_{ij}-v_{jk}+v_{il}-v_{ik}-v_{kl}$, whereas the right-hand side gives $v_{il}-v_{ij}-v_{jl}+v_{jl}-v_{jk}-v_{kl}$, which are both equal to $v_{il}-v_{ij}-v_{jk}-v_{kl}$.

\begin{theorem}\label{iso} The mapping $V_0$ induces a graded poset isomorphism $V$ from the set of circular permutations in $S_n$ into the set of admitted vectors in $\mathbb N^T$, ordered componentwise. The inverse mapping is completely determined as follows: let $w=1w'\in S_n$, viewed as word; then $V_0(w)=v$ implies that for any $1\leq i<j\leq n$, one has $\gamma_{ij}(w)=\delta_{1ij}(v)$.
\end{theorem}

\begin{figure}[h!]
\begin{center}
\includegraphics[scale=0.6]{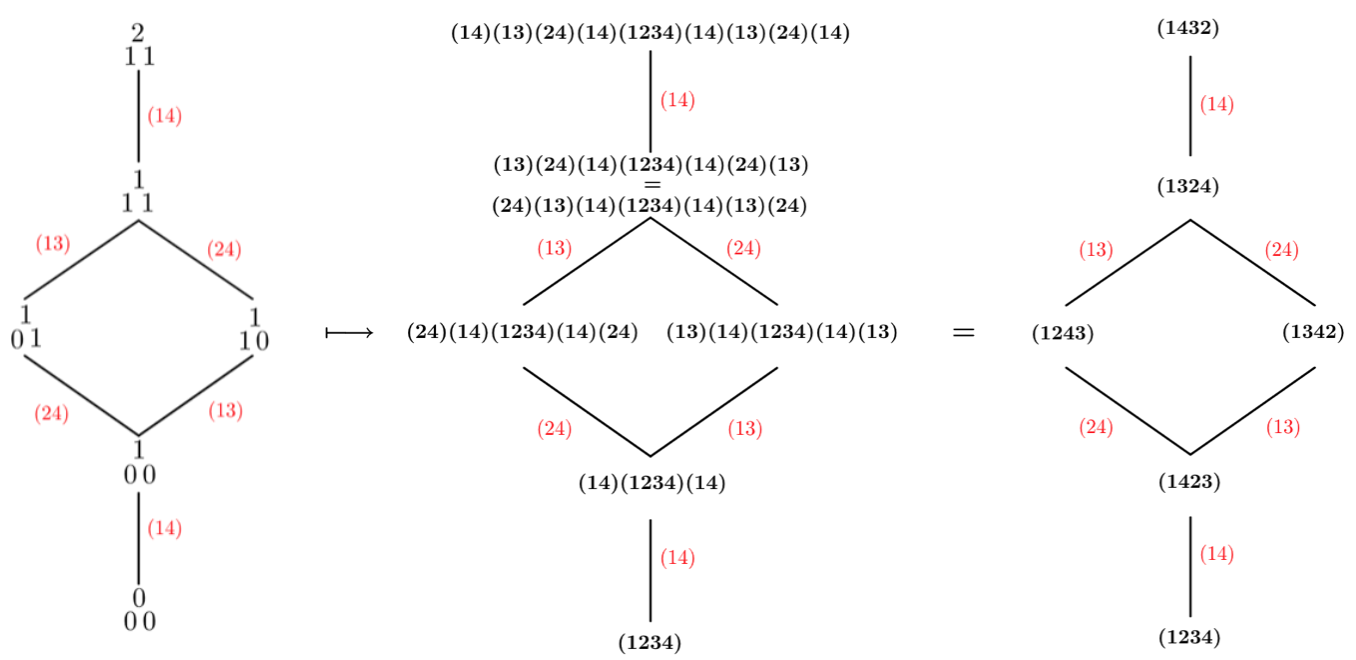} 
\end{center}
\caption{Poset isomorphism between the set of admitted vectors with $n=4$, and the set of circular permutations of $S_4$. The red labels represent, from left to right, the cover relation in the natural order on $\mathbb{N}^6$; the conjugation action; the large circular descents.}
\end{figure}

%See the right parts in Figures \ref{P4} and \ref{P5}.

The previous statement uses implicitly the well-known result that a permutation $w\in S_n$ is completely determined by its {\em sequence of inversions} $(\gamma_{ij}(w))_{1\leq i<j\leq n}$.

In order to prove the theorem, we need the following lemmas.

\begin{lemma}\label{cover} If $u,v$ are admitted vectors such that $u<v$, then for some $1\leq i<j\leq n$, one has
$$
\begin{array}{c}
i <p<j \Rightarrow \delta_{ipj}(u)=0; \\
1\leq p<i \Rightarrow \delta_{pij}(u)=1;  \\
j<p \leq n\Rightarrow \delta_{ijp}(u)=1.
\end{array}
$$
\end{lemma}

\begin{proof}%[Proof of Lemma \ref{cover}] 
%This follows from \cite{C}.
In the whole proof we use the fact that the functions $\delta$ map $u$ and $v$ onto 0 or 1.

1. We show first that there exists a pair $i<j$ such that $\delta_{ipj}(v) =0$ for any $i<p<j$. We construct below a nonempty set $E$, whose minimal elements, for a certain order, are such pairs.

The order on the set of pairs $(i,j)$,  $i<j$, is as follows: $(i,j)\leq (i',j')$ if the interval $[i,j]$ is contained in the interval $[i',j']$ (that is, if $i'\leq i$ and $j\leq j'$).

Consider the set $E$ of pairs $(i,j)$, $i<j$, such that $u_{ij}<v_{ij}$. Note that by assumption $E$ is nonempty. Let $(i,j)
$ be a minimal element in this set, for the previous order. Then for any $(i',j')<(i,j)$, one has by minimality $u_{i',j'}
=v_{i',j'}$; in particular, for $i<p<j$, $u_{ip}=v_{ip}$ and $u_{pj}=v_{pj}$. Moreover by construction, $u_{ij}-v_{ij}\leq -1$. Since $\delta_{ipj}(v)\leq 1$, it follows that:
\begin{align*}
\delta_{ipj}(u)& =u_{ij}-u_{ip}-u_{pj} \\
					& = u_{ij}-v_{ij}+v_{ij}-v_{ip}-v_{pj} \\
					&=(u_{ij}-v_{ij})+\delta_{ipj}(v) \\
					& \leq 0.
\end{align*}
Thus $\delta_{ipj}(u)=0$.

2. We consider now that set $F$ of pairs $(i,j)$, $i<j$, such that for any $i<p<j$, one has $\delta_{ipj}(u)=0$. This set is 
nonempty by 1. Let $(i,j)$ be a maximal element in this set. We claim that for any $1\leq p<i$, one has $\delta_{pij}(u)=1$ and 
for any $n\geq p>j$, $\delta_{ijp}(u)=1$. The claim implies the lemma.

Suppose by contradiction that the claim is not true; then by symmetry of the two cases, we may assume that for some $n\geq l>j$, one has $\delta_{ijl}(u)=0$, and we choose $l$ minimum. 

Let $j<k<l$. By minimality of $l$, we have $\delta_{ijk}(u)=1$. By Eq. (\ref{ptolemee}), $\delta_{ijk}(u)+\delta_{ikl}(u)=\delta_{ijl}(u)+\delta_{jkl}(u)$. Since $\delta_{ijk}(u)=1$ and $\delta_{ijl}(u)=0$, we must have $\delta_{ikl}(u)=0$. 

Now let $i<p<j$. Then by Eq. (\ref{ptolemee}) applied to $i<p<j<l$, we have $\delta_{ipj}(u)+\delta_{ijl}(u)=\delta_{ipl}(u)+\delta_{pjl}(u)$. Since $\delta_{ijl}(u)=0$ and $\delta_{ipj}(u)=0$, we must have $\delta_{ipl}(u)=0$.

We conclude that for any $p$ with $i<p<l$, we have $\delta_{ipl}(u)=0$. Since $(i,j)<(i,l)$, this contradicts the maximality of the pair $(i,j)$.
\end{proof}

Before stating and proving the next lemma, we note that the inequalities (\ref{admitted}) characterizing admitted vectors are 
equivalent to the condition that all functions $\delta_{ijk}$ maps admitted vectors into $\{0,1\}$.

\begin{lemma}\label{monter} Let $v$ be an admitted vector, let $i<j$ and define $v'$ by $v_{ij}+1=v'_{ij}$, whereas $v_{rs}=v'_{rs}$ for the other coordinates. Then the two following conditions are equivalent:

(i) $v'$ is an admitted vector; 

(ii) the following equations hold:
\begin{equation}\label{nathan}
\begin{array}{c}
i <p<j \Rightarrow \delta_{ipj}(v)=0; \\
1\leq p<i \Rightarrow \delta_{pij}(v)=1;  \\
j<p \leq n\Rightarrow \delta_{ijp}(v)=1.
\end{array}
\end{equation}
\end{lemma}

\begin{proof} Note that 
$$\begin{array}{c}
i<p<j \Rightarrow \delta_{ipj}(v)=v_{ij}-v_{ip}-v_{pj}=v'_{ij}-1-v'_{ip}-v'_{pj}=\delta_{ipj}(v')-1; \\
1\leq p<i \Rightarrow \delta_{pij}(v)=v_{pj}-v_{pi}-v_{ij}=v'_{pj}-v'_{pi}-v'_{ij}+1=\delta_{pij}(v')+1; \\
n\geq p>j \Rightarrow \delta_{ijp}(v)=\delta_{ijp}(v')+1;
\end{array}$$

In all other cases, $\delta_{rst}(v)=\delta_{rst}(v')$.

(i) implies (ii): let $i<p<j$; then by the above calculations $\delta_{ipj}(v)=\delta_{ipj}(v')-1$; since the values of $\delta$ are 0 or 1, we must have $\delta_{ipj}(v)=0$. Let $p<i$; then $\delta_{pij}(v)=\delta_{pij}(v')+1$; this forces $\delta_{pij}(v)=1$. Similarly, $p>j$ implies $\delta_{ijp}(v)=1$. Hence equations (\ref{nathan}) hold.

(ii) implies (i): It is enough to prove that  the functions $\delta$ take the values 0 or 1 when evaluated on $v'$. Let $i<p<j$; then $
\delta_{ipj}(v)=\delta_{ipj}(v')-1$, hence $\delta_{ipj}(v')=1$ by Eq.(\ref{nathan}). Let $1\leq p<i$; then $\delta_{pij}(v)=\delta_{pij}(v')+1$, 
hence $\delta_{pij}(v')=0$. Similarly, for $n\geq p>j$, $\delta_{ijp}(v')=0$. In all other cases, $\delta_{rst}(v')=\delta_{rst}(v)=0$ or $1$. This concludes 
the proof.
\end{proof}

\begin{corollary}\label{cover-admitted}
Let $u<v$ be admitted vectors.
Then there exists an admitted vector $u'$ such that $u<u'\leq v$ and that for some $1\leq i<j\leq n$ one has $u_{ij}+1=u'_{ij}$, whereas $u_{rs}=u'_{rs}$ for the other coordinates. 
\end{corollary}

Note that this implies that necessarily $u'$ covers $u$ for the order of admitted vectors. The corollary could be given a geometric proof, in the spirit of \cite{C}, but it follows also directly from the two previous lemmas.

\begin{lemma}\label{inversionij} Let $w\in S_n$, viewed as a word. If $1\leq i<j\leq n$, then $ji$ is a factor of $w$ if and only if the following conditions hold:

(i) $\gamma_{ij}(w)=1$;

(ii) for any $i<p<j$, $\gamma_{ip}(w)+\gamma_{pj}(w)=1$;

(iii) for any $1\leq p<i$, $\gamma_{pi}(w)=\gamma_{pj}(w)$;

(iv) for any $n\geq p>j$, $\gamma_{ip}(w)=\gamma_{jp}(w)$.
\end{lemma}

\begin{proof} Condition (i) means that $ji$ is a subword of $w$. If $1\leq p<i$, then $\gamma_{pi}(w)=\gamma_{pj}(w)=0$ 
means that $p$ is at the left of both $i$ and $j$ in $w$, and $\gamma_{pi}(w)=\gamma_{pj}(w)=1$ means that $p$ is at their right; hence condition (iii) means that $p$ is not between $i$ and $j$ in $w$. Condition (iv) is similar in the case $n\geq p>j$. Condition (ii) means similarly that for $i<p
<j$, $p$ is not between $i$ and $j$.

Thus the four conditions together mean that $ji$ is a factor of $w$.
\end{proof}

To complete the picture, we give the following dual result, whose proof may be deduced from the proof of Lemma \ref{monter}:

\begin{lemma}\label{descendre} Let $v'$ be an admitted vector, let $i<j$ and define $v$ by $v_{ij}=v'_{ij}-1$, whereas $v_{rs}=v'_{rs}$ for the other coordinates. Then the two following conditions are equivalent:

(i) $v$ is an admitted vector; 

(ii) the following equations hold:
\begin{equation}\label{nathan'}
\begin{array}{c}
i <p<j \Rightarrow \delta_{ipj}(v')=1; \\
1\leq p<i \Rightarrow \delta_{pij}(v')=0;  \\
j<p \leq n\Rightarrow \delta_{ijp}(v')=0.
\end{array}
\end{equation}
\end{lemma}

\begin{proof}[Proof of theorem \ref{iso}] 1. By Proposition \ref{V-0}, $V$ is well-defined on the set of circular permutations. 

2. We show that, for $w=1w'$, and $V_0(w)=v$, for any $1\leq i<j\leq n$, one has $\gamma_{ij}(w)=-v_{ij}+v_{1j}-v_{1i}=\delta_{1ij}(v)$. 

In the case $i\neq 1$, we have indeed 
\begin{align*}
\delta_{1ij}(v) =& -v_{ij}+v_{1j}-v_{1i} \\
 = &~\gamma_{ij}(w)-\sum_{i\leq k<j}\gamma_{k,k+1}(w)
-\gamma_{1j}(w)+\sum_{1\leq k<j}\gamma_{k,k+1}(w)
+\gamma_{1i}(w) \\ 
& ~ - \sum_{1\leq k<i}\gamma_{k,k+1}(w).
\end{align*}

 The three summations cancel, and $\gamma_{1j}(w)=0=\gamma_{1i}(w)$, since $w=1w'$. Thus the previous sum is equal to $\gamma_{ij}(w)$. In case $i=1$, one has $$\delta_{1ij}(v)=-v_{1j}+v_{1j}-v_{11}=0=\gamma_{1j}(w).$$

3. This implies that the mapping $V$ is injective, and proves the formulas of the statement.

4. In order to prove surjectivity, let $v$ be an admitted vector. Note that by the inequalities (\ref{admitted}), for $i<j$ the numbers  $-v_{ij}+v_{1j}-v_{1i}=\delta_{1ij}(v)$ 
are equal to 0 or 1.
Consider the set $I$ of pairs $(i,j)$ such that $1<i<j\leq n$ and $\delta_{1ij}(v)=1$. We show that this set is the {\em set of inversions by value} of some permutation $w$ (meaning that I is the set of $(i,j)$, $i<j$, such that $ji$ is a subword of $w$). 

Recall that a set $I$ of pairs $(i,j)$, $i<j$, is the set of inversions by value of a permutation if and only one has the two conditions:
(i) $(i,j)$ and $(j,k) \in I$ implies $(i,k)\in I$; and (ii) if $i<j<k$ and $(i,k)\in I$, then $(i,j)$ or $(j,k)$ is in $I$.

Therefore with $I$ as above, it is enough to prove (i) and (ii).

By Eq.(\ref{ptolemee}), we have (*) $\delta_{1ij}(v)+\delta_{1jk}(v)=\delta_{1ik}(v)(v)+\delta_{ijk}(v)$, and $\delta_{ijk}(v)=0$ or $1$. Therefore, (i) and (ii) 
easily follow. 

5. For the permutation $w$ constructed in 4., its set of inversions by value is $I=\{(i,j),1<i<j\leq n, \delta_{1ij}(v)=1\}$. Therefore $\gamma_{ij}(w)=-v_{ij}+v_{1j}-v_{1i}$ for any $1\leq i<j\leq 1$ (since it is also true for $i=1$, because by definition of $I$, there is no inversion $1j$).
%, and $\gamma_{1j}=0$ if $1<j$. 
We show that $V(w)=v$. We have indeed 
$$
-\gamma_{ij}(w)+\sum\limits_{i\leq k<j}\gamma_{k,k+1}(w)=v_{ij}-v_{1j}+v_{1i}+\sum\limits_{i\leq k<j}(-v_{k,k+1}+v_{1,k+1}-v_{1,k})=v_{ij},
$$
since $v_{k,k+1}=0$ and since the summation is a telescoping sum.

6. We have thus proved that $V$ is a bijection from the set of circular permutations onto the set of admitted vectors. We show now that $V$ is increasing. It is enough to show that if the permutation-word $w$ has the circular factor $ji$ with $j>i+1$ and if $w'$ is obtained from $w$ by exchanging $i$ and $j$, then $v'=V_0(w')$ is obtained by increasing by 1 the component $ij$ of $v=V_0(w)$. Since $V_0$ is invariant under conjugation of $w$, we may assume that $ji$ is a factor of $w$.

We have 
$$
v_{ij}=-\gamma_{ij}(w)+\sum\limits_{i\leq k<j}\gamma_{k,k+1}(w),
$$ 
and 
$$
v'_{ij}=-\gamma_{ij}(w')+\sum\limits_{i\leq k<j}\gamma_{k,k+1}(w').
$$

Thus $v_{ij}+1=v'_{ij}$, since the values of $\gamma_{k,k+1}$ are the same for $w$ and $w'$ and since the subword $ji$ disappears from $w$ to $w'$. Let now $l>k+1$ and $(k,l)\neq (i,j)$. Similar arguments then show that $v_{kl}=v'_{kl}$.

7. To conclude, we must show that if $v=V_0(w)$, $v'=V_0(w')$ and $v<v'$, then $(w)<(w')$ for the order of circular permutations. In view of Corollary \ref{cover-admitted}, we are reduced to the case where $v$ and $v'$ differ by 1 on some coordinate $ij$: $v'_{ij}=v_{ij}+1$ and $v_{kl}=v'_{kl}$ if $(k,l)\neq (i,j)$.

We may assume that $w=1u$, and we show that $ji$ is a circular factor of $w$, which will suffice, by definition of the order on circular permutations and by 6. 

Suppose that $i>1$. We apply Lemma \ref{inversionij} to show that $ji$ is a factor of $w$. We apply below 
several times Eq.(\ref{ptolemee}), Lemma \ref{monter} and 2.

(i) We have $\gamma_{ij}(w)=\delta_{1ij}(v)=1$. 

(ii) Let $i<p<j$. Then we have
$$
\gamma_{ip}(w)+\gamma_{pj}(w)=\delta_{1ip}(v)+\delta_{1pj}(v)=\delta_{1ij}(v)+\delta_{ipj}(v)=1.
$$

(iii) Let $p<i$. Then 
$$
\gamma_{pi}(w)-\gamma_{pj}(w)=\delta_{1pi}(v)-\delta_{1pj}(v)=-\delta_{1ij}(v)+\delta_{pij}(v)=0.
$$ 

(iv) is similar. We conclude that $ji$ is a factor of $w$.

Suppose now that $i=1$. Then we have to show that $j1$ is a circular factor of $w=1u$, which means that $j$ is the 
last letter of $w$. Thus we must prove that for $1<p<j$, $jp$ is not a subword of $w$, that is $\gamma_{pj}(w)=0$; and if $p>j$, that $jp$ is a not a subword, that is, $pj$ is one, equivalently $\gamma_{jp}(w)=1$. But we have, using Lemma \ref{monter}:
if $p<j$, $\gamma_{pj}(w)=\delta_{1pj}(v)=0$; and if $p>j$, $\gamma_{jp}(w)=\delta_{1jp}(v)=1$, which concludes the proof.
\end{proof}

The proof implies the following result, which is illustrated by the labels of the edges in the Hasse diagram of Figures \ref{PyraA4} and \ref{PyraA5}. 

\begin{corollary}\label{label} Let $\sigma,\sigma'$ be circular permutations and $v,v'$ be the corresponding admitted vectors. Then $\sigma'$ covers $\sigma$ if and only if $v'$ covers $v$. Moreover the following conditions are equivalent, for $1\leq r, r+1<s\leq n$:

(i) $\sigma'$ is obtained from $\sigma$ by replacing the circular factor $sr$ by $rs$;

(ii) $v'_{rs}=v_{rs}+1$.
\end{corollary}

In the next result, we transfer to the poset of admitted vectors the anti-automorphisms of the poset of circular permutations, as they appear in Corollary \ref{inversion}.

\begin{corollary}\label{anti-autom} 1. The anti-automorphism of the poset of admitted vectors, corresponding to 
inversion 
in the poset of circular permutations, is defined by $ u\mapsto v$, with $v_{ij}=j-i-1-u_{ij}$.

2. The anti-automorphism of the set of admitted vectors, corresponding to conjugation by the longest permutation 
$w_0$, is defined by $u\mapsto v$, with $v_{ij}=j-i-1-u_{n+1-j,n+1-i}$.
\end{corollary}

Note that the product of the two previous anti-automorphisms of the poset is the automorphism $u\to v$ with $v_{ij}
=u_{n+1-j,n+1-i}$. 

\begin{proof} 1. Inversion of circular permutations is the mapping $(w)\mapsto (\tilde w)$; the corresponding admitted vectors are respectively $u$ and $v$ 
with

$$
u_{ij}=-\gamma_{ij}(w)+\sum_{i\leq k<j}\gamma_{k,k+1}(w),
$$
 and 
 $$
 v_{ij}=-\gamma_{ij}(\tilde w)+\sum\limits_{i\leq k<j}\gamma_{k,k+1}(\tilde w).
 $$
 Since $\gamma_{rs}(\tilde w)=1-\gamma_{rs}(w)$, we have 
$$
v_{ij}=-(1-\gamma_{ij}(w))+\sum\limits_{i\leq k<j}(1-\gamma_{k,k+1}(w))=-u_{ij}-1+j-i.
$$

2. This anti-automorphism is defined, for circular permutations, by $(w)\mapsto (w')$, where $w'$ is obtained by replacing 
in $w$ each letter by its {\em complement} $n+1-k$. Hence $\gamma_{ij}(w')=1-\gamma_{n+1-j,n+1-i}(w)$. Therefore, denoting by $u,v$ respectively the admitted vectors corresponding to $(w)$ and $(w')$, we have 
\begin{align*}
v_{ij} & = -\gamma_{ij}(w')+\sum_{i\leq k<j}\gamma_{k,k+1}(w') \\
	    & = -(1-\gamma_{n+1-j,n+1-i}(w))+\sum\limits_{i\leq k<j}(1-\gamma_{n+1-
k-1,n+1-k}(w)) \\
        & = -1+j-i-(-\gamma_{n+1-j,n+1-i}(w)+\sum\limits_{l=n+1-j}^{n+1-i}\gamma_{l,l+1}(w)) \\
        & = j-i-1-u_{n+1-j,n+1-i}.
\end{align*}
\end{proof}

\begin{corollary}\label{largest} The smallest admitted vector is the null vector, and the largest one, $w$ say, satisfies $w_{ij}=j-
i-1$; moreover, $\delta_{ijk}(w)=1$ for any $i<j<k$.
\end{corollary}

\begin{proof} Clearly, the null vector is admitted and is the smallest element in the poset of admitted vectors. Using Corollary \ref{anti-autom}, we see that $w$ as defined in the statement is the largest element. Now 
\begin{align*}
\delta_{ijk}(w)& = w_{ik}-w_{ij}-w_{jk} \\
					& = k-i-1-(j-i-1)-(k-j-1) \\
					& = 1.
\end{align*}
\end{proof}

\begin{proof}[Proof of Proposition \ref{path}]
1. It is enough to prove this independence result when $\sigma=(12\cdots n)$. Consider a path from $\sigma$ to $
\tau$, whose successive edges are labelled $(r_1,s_1),\ldots, (r_p,s_p)$, with $r_i+1<s_i$, and let $\sigma=(w_0),(w_1), 
\ldots,(w_{p-1}),(w_{p})=\tau$ the successive vertices in the path, written as $n$-cycles, so that each $w_i$ is a word. 
Then $s_ir_i$ is a circular factor of the word $
w_{i-1}$, and $w_i$ is obtained from the latter word by exchanging $r_i$ and $s_i$ in it.

Define the permutations $\alpha_i=(r_i,s_i)\circ\cdots\circ(r_1,s_1)$. Then $(w_i)=\alpha_i\circ\sigma\circ\alpha_i^{-1}$. It follows that, in cycle notation, $
(w_i)=(\alpha_i(1)\cdots\alpha_i(n))$ and therefore $s_{i+1}r_{i+1}$, being a circular factor of $(w_{i})$, is a circular factor of 
the word $\alpha_i$. We have $\alpha_{i+1}=(r_{i+1},s_{i+1})\circ\alpha_i$ and thus, as words, $\alpha_{i+1}$ is obtained from $\alpha_i$ by 
exchanging $r_{i+1}$ and $s_{i+1}$, whose positions moreover must differ by 1 in both words (modulo $n$). Hence:
$$
 \left\{
                          	\begin{array}{rl}
 						\alpha_{i+1}^{-1}(r_{i+1})\equiv\alpha_i^{-1}(r_{i+1})-1\mod n \\
 						\alpha_{i+1}^{-1}(s_{i+1})\equiv\alpha_i^{-1}(s_{i+1})+1\mod n \\
  				      	\alpha_{i+1}^{-1}(j)=\alpha_i^{-1}(j) ~~\text{if}~~j \neq r_{i+1}, s_{i+1}.\\
					    \end{array}
					    \right.
$$

It follows by induction that for any $j=1,\ldots,n$, the position of $j$ in the word $\alpha_{i+1}$ is congruent modulo $n$ to: $j$ {\em plus} the 
number of $(rs)$ on the path with $s=j$ {\em minus} the number of $(rs)$ on the path with $r=j$. From Corollary \ref{label}, it follows that it is congruent modulo $n$ to:
$$
j+\sum_{r<s=j}v_{rs}-\sum_{j=r<s}v_{rs}
$$
 where $v$ is the admitted vector associated to the circular permutation $(w_{i+1})$.

Therefore each $\alpha_{i+1}$ depends only on $v$, which depends only  on $(w_{i+1})$; thus $\alpha=\alpha_{p}$ depends only on $\tau$.

2. It remains to determine $\alpha$ for a maximal path in the poset. The previous calculations and Corollary 
\ref{largest} show that $
\alpha^{-1}(j)$ is congruent modulo $n$ to $j+\sum\limits_{r<j}(j-r-1)-\sum\limits_{j<s}(s-j-1)$. Moreover, a direct computation shows that:
\begin{align*}
&j+\sum_{r<j}(j-r-1)-\sum_{j<s}(s-j-1)  \\
 = &~j+((j-2)+\cdots+1+0)-(0+1+\cdots+(n-j-1))\\
 =  &~j+\frac{(j-2)(j-1)}{2}-\frac{(n-j-1)(n-j)}{2}.
\end{align*}

 The cases $j=1$ and $j=n$ have to be treated separately, but the final formula is correct in these cases as well. 

Suppose that $n=2k+1$. Then using the fact that the inverse of $2 \mod n$ is $-k$, we find that $\alpha^{-1}(j)$ is 
congruent modulo $n$ to 
$j-k(j-2)(j-1)+kj(j+1)$. Hence, omitting the calculations, to $n+1-j$. Thus, as words, $\alpha^{-1}=n\cdots21$, and $\alpha=\alpha^{-1}$, since it is an 
involution.

Suppose now that $n=2k$. Then a calculation left to the reader shows that modulo $n$, $\alpha^{-1}(j)$ is congruent to  $k+1-j$. Hence, as word, $\alpha^{-1}=(n/2)\cdots 21 n\cdots (1+n/2)$ and $\alpha=\alpha^{-1}$ 
since it is an involution.
\end{proof}

\section{Properties of the poset}

\subsection{Lattice}

Recall that the the set of admitted vectors in $\N^T$ has been defined if the previous section.

\begin{theorem}\label{lattice} The poset of admitted vectors (or equivalently the poset of circular permutations) is a lattice.
\end{theorem}

We give a combinatorial proof of this result, with a simple algorithmic construction of the supremum and the infimum. This result may also be deduced from Theorem \ref{interval}: see the comment before Corollary \ref{semi-distrib}.

%Further we deduce from a result of Nathan Reading and David E Speyer, that this lattice is semidistributive (see Corollary \ref{semi-distrib} and the remark before it).

%Notice that this result has been proven by the second-named author for any affine Weyl group $W_a$ in \cite{C}, by proving that the poset is isomorphic to an interval for the right weak order of $W_a$.

We extend the definition of admitted vector as follows: if $v\in \N^T$, and $i<k$, then we say that {\em the component $v_{ik}
$ of $v$ is  admitted} if either $k=i+1$ and $v_{ik}=0$ or $k>i+1$ and 
$$\forall j, i<j<k \Rightarrow v_{ij}+v_{jk} \leq v_{ik} \leq v_{ij}+v_{jk}+1.$$

Hence $v$ is admitted if and only if all its components are admitted.

\begin{lemma} If for $v\in \N^T$ all components $v_{ik}$ with $k-i<n-1$ are admitted, then the numbers $v_{1i}+v_{in}$, $i=2,\ldots,n-1$ take at most two values, which are consecutive.
\end{lemma}

\begin{proof}
By hypothesis, we have for $1<i<j<n$, $v_{1j}=v_{1i}+v_{ij}+\eta_{1ij}$ and $v_{in}=v_{ij}+v_{jn}+\eta_{ijn}$ 
for some  $\eta_{1ij},\eta_{ijn}\in\{0,1\}$. It follows that 
$$
v_{1i}+\eta_{1ij}+v_{in}=v_{1j}-v_{ij}+v_{ij}+v_{jn}+
\eta_{ijn}
=v_{1j}+v_{jn}+\eta_{ijn}.
$$
 Since the $\eta$'s take only the values 0 or 1, it follows that any two of the numbers 
$v_{1i}+v_{in}$, $i=2,\ldots,n-1$, differ in absolute value by at most 1. It follows easily that they take at most two values, 
which must be consecutive.
\end{proof}

By appropriately reindexing the vector, we obtain the following corollary.

\begin{corollary} Suppose that $v\in \N^T$ and that, for some $i<j$, the components $v_{rs}$ with $s-r<j-i$ are 
admitted. Then the numbers $v_{ip}+v_{pj}$, $p=i+1,\ldots,j-1$, take at most two values, which are consecutive.
\end{corollary}

\begin{proof}[Proof of Theorem \ref{lattice}]
0. Let $u,v$ be two admitted vectors. We define $X\in \N^T$ as follows: $X_{ij}$ is defined by induction on $j-i\geq 1$; if $j=i+1$, $X_{ij}=0$; if $j>i+1$, then 
$X_{ij}=\max\{u_{ij},v_{ij},X_{ip}+X_{pj}, i<p<j\}$.

1. We show first that $X$ is an admitted vector. It is enough to show that each coefficient $X_{ij}$ is admitted, and we do it by induction on $j-i$. The case $j-i=1$ is clear, so take $j-i=h>1$. By induction, all coefficients $X_{rs}$ with $s-r<h$ are admitted. 

Suppose that $X_{ij}=X_{ip}+X_{pj}$ for some $p$ with $i<p<j$. By the corollary, one has 
$\{X_{iq}+X_{qj} , i<q<j\} \subset \{a-1,a\}$, and by the definition of $X_{ij}$, we must have   $a=X_{ip}+X_{pj}$. It follows that for any $q$, $i<q<j$, one 
has $X_{iq}+X_{qj}=X_{ij}$ or $X_{ij}-1$ and therefore $X_{iq}+X_{qj}\leq X_{ij}\leq X_{iq}+X_{qj}+1$. 

If $X_{ij}$ is not equal to $X_{ip}+X_{pj}$ for some $p$ with $i<p<j$, then we must have (choosing one of the two 
similar cases) $X_{ij}=v_{ij}\geq u_{ij}$ and $v_{ij}>X_{ip}+X_{pj}$ for any $p$ with $i<p<j$. Then $X_{ip}+X_{pj}<X_{ij}$; moreover $v$ being admitted, $X_{ij}=v_{ij}\leq v_{ip}+v_{pj}+1\leq X_{ip}+X_{pj}+1$, since by the recursive construction of $X$,
$v_{rs}\leq X_{rs}$ for any $rs$ with $s-r<j-i$.

We conclude that in all cases, $X_{ij}$ is admitted, and therefore  by induction $X$ is admitted.

2. We show now that $X$ is the supremun of $u$ and $v$, that is, $X$ is the smallest element in the set of upper 
bounds of $u,v$. By construction, $X_{ij}\geq u_{ij},v_{ij}$, so that $X\geq u,v$ and $X$ is an upper bound of $u$ and 
$v$. Now let $Y$ be an admitted vector which is an upper bound of $u,v$. We show by induction on $j-i$ that $X_{ij}\leq Y_{ij}$, which will imply 
that $X\leq Y$, as was to be shown. We have $X_{ij}=0=Y_{ij}$ if $j-i=1$. Suppose that $j-i>1$. Then $Y_{ij}\geq 
u_{ij},v_{ij} $; and by the inequalities (\ref{admitted}), if $i<p<j$, $Y_{ij}\geq Y_{ip}+Y_{pj}$ which by induction is $\geq 
X_{ip}+X_{pj}$; thus $Y_{ij}\geq \max\{u_{ij},v_{ij},X_{ip}+X_{pj}, i<p<j\}= X_{ij}$.

3. Thus $u\vee v$ exists in the poset. Since the poset has an anti-automorphism, $u\wedge  v$ exists, too, and the poset is a lattice.
\end{proof}

The proof allows to compute $u\vee v$. A direct computation of $X=u\wedge v$ is obtained recursively, by induction on $j-i$, as follows: $X_{i,i+1}=0$; if $j>i+1$, then 
$X_{ij}=\min\{u_{ij},v_{ij},X_{ip}+X_{pj}+1, i<p<j\}$. The proof is left to the reader (one may use the first anti-automorphism of 
Corollary \ref{anti-autom}). 

The lattice is not modular. Indeed, looking in Figure \ref{P5}, it is seen that the 
circular permutations $(14235)$ and $(13425)$ have 
infimum $(14253)$ and supremum $(13542)$, and respective rank 4,4,3 and $6$. 
Since in a modular lattice, the rank function $\rho$ satisfies 
$\rho(x)+\rho(y)=\rho(x\wedge y)+\rho(x\vee y)$ (see \cite{S} page 104), the lattice is not 
modular. Since each 
distributive lattice is 
modular (\cite{S} p. 106), the lattice is not distributive either.

\subsection{Multiplicities in the Hasse diagram and Eulerian numbers}

Recall that the {\em Eulerian numbers} are defined as follows: $a(n,k)$ is the number of permutations in $S_n$ whose 
descent set has $k$ elements, where the {\em descent set} of $\sigma\in S_n$ is $\{i,1\leq i\leq n-1,\sigma(i)>\sigma(i+1)\}$. These numbers satisfy the recursion $$a(n,k)=(k+1)a(n-1,k)+(n-k)a(n-1,k-1)$$ if $n,k\geq 1$, with the initial conditions $a(0,k)=0$ for any $k\geq 1$, and $a(n,0)=1$ for any $n\geq 0$. See \cite{P}.

%\begin{lemma} 
%\end{lemma}

\begin{definition}\label{large circular permutation}Let $\sigma=(a_1\cdots a_n)$ be a circular permutation in $S_n$. We call {\em large circular descent} of $\sigma$ a letter $a_i$ such that $a_{i-1}>a_i+1$, were the indices are taken modulo $n$. The number of large circular descents of $\sigma$ is denoted by $d(\sigma)$.

Symmetrically, we call {\em large circular ascent} of $\sigma$ a letter $a_i$ such that $a_{i}<a_{i+1}+1$, were the indices are taken modulo $n$.
\end{definition}

In other words, $b$ is a large circular descent if, denoting by $a$ the letter before $b$ in the cycle, one has $a>b+1$. Note that we say that $b$ is the large circular descent, and not $a$; the reason for this notational shift will appear in the proofs below. For example, the large circular descents of $(1,4,2,6,5,3)$ are $1$ (because before 1 there is 3, cyclically speaking), 2 (because before 2 there is 4), and $3$; 5 is not a large circular descent. 

\begin{theorem}\label{euler} The number of circular permutations in $S_{n+1}$ having $k$ large circular descents is 
equal to the Eulerian number $a(n,k)$.
\end{theorem}

\begin{corollary}\label{coverk} In the poset of circular permutations in $S_{n+1}$, the number of elements which are covered by $k$ 
elements is $a(n,k)$. In particular the number of inf-irreducible elements is $a(n,1)=2^n-n-1$. 
\end{corollary}

The next corollary is somewhat curious, since the two poset are not isomorphic, nor are their Hasse diagram as graphs.

\begin{corollary}
The total number of edges in the Hasse diagram of the poset of circular permutations in $S_{n+1}$ is equal to the total 
number of edges in the Hasse diagram of the poset of the weak order in $S_n$.
\end{corollary}

\begin{proof} This follows from the theorem since the edges in the Hasse diagram of the right weak order are in 
bijection with the set of descents of all permutations.
\end{proof}

To prove the theorem, we show that the number of circular permutations in $S_{n+1}$ with $k$ large circular descents satisfy the same recursion as the Eulerian numbers, and the same initial conditions. Finding a bijective proof is an open question.

\begin{proof}[Proof of Theorem \ref{euler}]

1. Denote by $C(n)$ the set of circular permutations in $S_{n+1}$, so that $|C(n)|=n!$. Denote by $C(n,k)$ the set 
of elements in $C(n)$ having $k$ large circular descents, and $c(n,k)=|C(n,k)|$. 

We verify first that 
the initial conditions are the same. Indeed, $c(0,k)$ for $k\geq 1$ is the number of circular permutations in $S_1$ having $k$ large circular descents, hence $c(0,k)=0$. Moreover, for $n\geq 0$, $c(n,0)$ is the number of circular permutations in $S_{n+1}$ having no large circular descent; this permutation is $(n,n-1,\ldots,2,1)$ as we saw in the proof of Corollary \ref{order1}, and therefore $c(n,0)=1$.

We show below that 
$c(n,k)=(k+1)c(n-1,k)+(n-k)c(n-1,k-1)$, which is the same recursion than the Eulerian numbers. Thus we will obtain $c(n,k)=a(n,k)$, which proves the theorem.

2. Recall that each element in $C(n)$ is represented by a circular word $(a_1,\ldots,a_{n+1})$ and that we do not distinguish between all its cyclic rearrangements (so that e.g.($a_1,\ldots,a_{n+1})=(a_2,\ldots,a_{n+1},a_1)$).

3. Define a mapping $\phi$ from $C(n)$ into $C(n-1)\times [n]$ as follows. Associate with 
$\alpha=(a_1,\ldots,a_{n+1})\in C(n)$ the circular permutation $\beta=(b_1,\ldots,b_n)\in C(n-1)$ obtained by removing $n+1$ from the 
cycle, and let $b_j$ denote the letter such that $(a_1,\ldots,a_{n+1})$ is obtained from $\beta=(b_1,\ldots,b_n)$ by 
inserting $n+1$ before $b_j$. 

Examples ($n+1=7$): a) $\alpha=(2,4,6,5,1,7,3)\in C(6)$, $\beta=(2,4,6,5,1,3)\in C(5)$, $b_j=3\in [6]$. b) $\alpha=(2,4,6,5,1,3,7)$, $
\beta=(2,4,6,5,1,3)$, $b_j=2$.

We define $\phi(\alpha)=(\beta,b_j)$. This is clearly a bijection, since we recover 
$\alpha$ from $\beta$ by inserting $n+1$ before $b_j$, so that $\phi$ is injective, and since the two sets have the same 
cardinality.

4. For $\beta\in C(n-1)$, define $D(\beta)$ to be the set of large circular descents of $\beta$ together with the letter $n$; hence $|D(\beta)|=d(\beta)+1$, since $n$ cannot be a large circular descent. Moreover, let $E(\beta)$ be the set $[n]\setminus D(\beta)$, so that $|E(\beta)|=n-1-d(\beta)$.

We claim that the restriction of the mapping $\phi$ to $C(n,k)$ is a bijection from 
$C(n,k)$ onto
$$\{(\beta,b),\beta\in C(n-1,k), b\in D(\beta)\}\cup\{(\beta,b),\beta\in C(n-1,k-1), b\in E(\beta)\}.$$
The set $C(n,k)$ has $c(n,k)$ elements, and the displayed set has $c(n-1,k)(k+1)+c(n-1,k-1)(n-k)$ elements. Thus the recursion follows.

5. In order to prove the claim, we prove that if $\phi(\alpha)=\beta$, then $\alpha$ and $\beta$ have the same number of large circular descents if and only if $b_j$ is a large 
circular descent of $\beta$ or if $b_j=n$; and that otherwise, $\alpha$ has one more. 

Indeed, suppose that $b_j$ is large circular descent of $\beta$, and in particular $b_j<n$; then inserting $n+1$ before $b_j$ creates in $\alpha$ the large circular 
descent $b_j$ (since $n+1>b_j+1$), but $n+1$ is not a large circular in $\alpha$; hence $d(\alpha)=d(\beta)$. Suppose now that 
$b_j=n$; then the previous insertion does not create a new large circular descent, nor remove one; we have $d(\alpha)=d(\beta)$, too.

Suppose now that $b_j$ is not a large circular descent of $\beta$, nor $n$; then the insertion of $n+1$ before $b_j$ 
creates a new large circular descent; thus $d(\alpha)=d(\beta)+1$.
\end{proof}

\subsection{Limiting poset}

Suppose that we have a family $P_n$, $n\in\N$, of ranked posets; suppose also that for any $k$, and for $n$ large enough, the subposets $\{x\in P_n, rk(x)\leq k\}$ are all isomorphic. We may then define the {\em limit} of these posets as the ranked poset $P$, whose elements of rank $\leq k$ form the poset $\{x\in P_n, rk(x)\leq k\}$, $n$ large enough.

Recall that a {\em partition} of $n$ is a multiset of positive integers whose sum is $n$ ($n$ is the {\em weight} of the partition); a partition of $n+1$ {\em covers} 
a partition of $n$ if the latter is obtained from the latter by adding 1 to the multiset, or by replacing some element $a$ 
by $a+1$. This is {\em Young's lattice}, with rank equal to the weight function; see \cite{Sa} Definition 5.1.2. 

\begin{theorem}\label{limit} The posets of circular permutations have a limit poset, which is Young's lattice.
\end{theorem}

\begin{proof} 1. We prove first that the limit exists. Let $k$ be an integer. Suppose that $n\geq 2k$. We show that the 
subposet of elements of rank $\leq k$ is independent of $n$. 

For this, recall from Section \ref{circular} how the poset is constructed: one represents each circular permutation by one of its cyclic representative, which is a permutation-word.  Each covering relation $w\to w'$ is  of the form either $w=usrv,w'=ursv$, or $w=rus,w'=sur$, $s>r+1$. In the first case we call the covering relation {\em inner}, and in the second case {\em outer}.

The minimal circular permutation is $(1,2,\ldots,n)$, and we choose to represent it by the word $w_0=(\lfloor n/2\rfloor+1)\dots(n-1)n12\ldots \lfloor n/2\rfloor$. For example, in the cases $n=5,6$, it is respectively $34512$ and $456123$. 

We then generate the elements of the poset of rank $\leq k$. We therefore use at most $k$ covering relations. A 
moment's thought shows that, since the first covering relation is necessarily determined by the middle factor $n1$, and since $k\leq \lfloor n/2\rfloor$, we never perform an outer covering relation. And this 
implies that the generated subposet is independent of $n$.

2. To finish the proof, it is enough to prove that if $n\geq 2k$, then the subposet of elements of rank $\leq k$ is 
isomorphic with the subposet of the Young lattice whose elements are the partitions of weight $\leq k$.

By the previous part, all these former elements are a shuffle of the two words $(\lfloor n/2\rfloor+1)\dots(n-1)n$ and $12\ldots 
\lfloor n/2\rfloor$, and the order is the inverse order of the right weak order on these words. 
%The rank of such a shuffle, $w$ say, is the number of the 
Note that each such shuffle, $w$ say, is completely determined by the sequence $a_1,a_2,a_3,\ldots,a_{\lfloor n/
2\rfloor}$ where $a_i$ is the number of $j=(\lfloor n/2\rfloor+1),\dots,(n-1),n$, which are at the right of $i$ in $w$. For example, with $n=6$, we have the shuffle $415263$, with $a_1=2, a_2=1,a_3=0$.

The 
rank of $w$ is the sum of the $a_i$'s. One has $a_1\geq a_2\geq a_3\ldots$, and therefore the sequence determines a 
partition of some integer $\leq k$. This sequence defines the isomorphism between the two posets. We leave the 
details to the reader.
\end{proof}

\section{The functions $\delta$ and triangulations of an $n$-gon}

We may characterize the sequences $(\delta_{ij}(v))_{1\leq i<j\leq n}$, $v$ admitted vector. The latter is completely determined by this sequence. The computation of $v$ leads to triangulations of an $n$-gon.

\begin{theorem} The mapping $v \mapsto (\delta_{ijk}(v))_{1\leq i<j<k\leq n}$ is a bijection from the set of admitted vectors 
onto the set of sequences $(a_{ij})_{1\leq i<j<k\leq n}$ with values in $\{0,1\}$, satisfying the relation $a_{ijk}+a_{ikl}=a_{ijl}+a_{jkl}$ for any 
$i<j<k<l$.
\end{theorem}

\begin{proof} The mapping is well-defined by Eq. (\ref{ptolemee}). Injectivity follows by induction on $j-i$ from the initial conditions $v_{i,i+1}=0$ and if $j-i>1$, from the equation $ v_{ij}=v_{ip}+v_{pj}+\delta_{ipj}(v)$ for any $p$, $i<p<j$.

To prove surjectivity, let $(a_{ij})_{1\leq i<j\leq n}$ be a sequence, with values in $\{0,1\}$, satisfying the relations $a_{ijk}+a_{ikl}=a_{ijl}+a_{jkl}$ for any 
$i<j<k<l$. Define $v_{ij}$ by induction: $v_{i,i+1}=0$ and for $j-i>1$, choose $p$ between $i$ and $j$, and define $v_{ij}$ by $v_{ij}=v_{ip}+v_{pj}+a_{ipj}$. 
We verify, by induction on $j-i$, that this does not depend on the choice of $p$. So let $i<p<q<j$. We have by induction 
$v_{iq}=v_{ip}+v_{pq}+a_{ipq}$ and $v_{pj}=v_{pq}+v_{qj}+a_{pqj}$. The above relations imply that $a_{ipq}+a_{iqj}
=a_{ipj}+a_{pqj}$. Hence 
\begin{align*}
v_{ip}+v_{pj}+a_{ipj} & = v_{ip}+v_{pq}+v_{qj}+a_{pqj}+a_{ipj} \\
							   & = v_{ip}+v_{pq}+v_{qj}+a_{ipq}+a_{iqj} \\
							   & = v_{iq}+v_{qj}+a_{iqj},
\end{align*}
 and so $v_{ij}$ is well-defined. The fact that $v$ is an admitted vector then follows from the fact that for any $i<j<k$, $v_{ik}-v_{ij}-v_{jk}=a_{ijk}\in \{0,1\}$, implying the inequalities (\ref{admitted}).
\end{proof}

Consider a convex $n$-gon in the plane whose vertices are sequentially numbered $1,2,\ldots,n$. In this $n$-gon, a {\em triangle} (resp. {\em convex quadrangle}) is determined by a subset of three (resp. four) vertices of the $n$-gon. A {\em triangulation} of the $n$-gon is a set of triangles whose interiors do not pairwise intersect, and who cover the whole $n$-gon. 

Note that a 3-subset $H$ of $\{1,2,\ldots,n\}$ determines uniquely three numbers $i<j<k$ such that $H=\{i,j,k\}$. We write $\delta_H:=\delta_{ijk}$.

\begin{theorem}\label{triang} If $v$ is an admitted vector, and $\mathcal T$ a triangulation of the $n$-gon, then $v_{1n}$ is equal to the sum of all $\delta_T(v)$ for all triangles $T\in \mathcal T$.
\end{theorem}

\begin{proof} Consider the unique triangle $\{1,2,p\}$ in $\mathcal T$ (it is the triangle containing the edge $12$ of the $n$-gon). By induction, we have $v_{1p}=\sum_T\delta_T(v)$, sum over all triangles in $\mathcal T$ contained in the sub-$p$-gon with vertices $1,2,\ldots,p$, which form a triangulation of this $p$-gon. Similarly, the triangles contained in the $n-p+1$-gon with vertices $p,p+1,\ldots, n$ form a triangulation of it; hence by induction (after appropriate reindexing), $v_{pn}=\sum_T\delta_T(v)$, sum over all triangles of this triangulation. If we add the the triangle with vertices $1,p,n$ to these two triangulations we obtain the original triangulation. Now $v_{1n}=v_{1p}+v_{pn}+\delta_{1pn}$, which concludes the proof.
\end{proof}

More generally, for any $i<j$, the coefficient $v_{ij}$ of an admitted vector $v$ is equal to $\sum_T\delta_T(v)$, where the sum is over all triangles of a triangulation of the $j-i+1$-gon with vertices $i,i+1,\ldots, j$.

Given a triangulation, consider a convex quadrangle such that one of its diagonal determines two triangles of the triangulation. If we replace in the triangulation these two triangles by the two others of the quadrangle, determined by the other diagonal, then we obtain a new triangulation: we say that the new triangulation is obtained by {\em mutation}. It is known that given two triangulations of the $n$-gon, one may transform one into the other by a sequence of mutations.

Note that the equation (\ref{ptolemee}) implies the fact that the sum $\sum_T\delta_T(v)$, over all triangles in a triangulation, is invariant under mutation.

We note finally that the sequence $(\delta_{ijk}(v))_{1\leq i<j<k\leq n}$ is completely defined by the subsequence $(\delta_{1ij}(v))_{i<j}$, as follows from Eq. (\ref{ptolemee}), by considering the quadrangle $1,i,j,k$. The latter sequence moreover appears as the inversion sequence of the permutation $w=1w'$ such that $V(w)=v$ (see the proof of Theorem \ref{iso}).

\section{An interval in the affine symmetric group}\label{-interval}

\subsection{Preliminaries on $\tilde S_n$}

We begin by recalling some definitions and results on the {\em affine symmetric group}. Details may be found in the book by Anders Bj\" orner and Francesco Brenti \cite{BB}. This group is denoted $\tilde S_n$. It is the group of permutations $f$ of $\mathbb Z$ satisfying the two conditions:
\begin{itemize}
\item $f(x+n)=f(x)+n$ for any integer $x$;
\item $f(1)+f(2)+\cdots+f(n)=n(n+1)/2$.
\end{itemize}

An element of $\tilde S_n$ is called an {\em affine permutation}. Each usual permutation in $S_n$ extends uniquely to an affine permutation, so 
that $\tilde S_n$ contains the symmetric group $S_n$ as a subgroup. The group $\tilde S_n$ is generated by the {\em adjacent affine 
transpositions} $s_0, s_1,\ldots,s_{n-1}$, defined by the condition that $s_i$ exchanges $i$ and $i+1$ if $i=0,\ldots, n-1$. 
Each affine permutation $f$ is represented by a {\em window} $[f(1),f(2),\ldots,f(n)]$ and we write $f=[f(1),f(2),\ldots,f(n)]$. For example in $\tilde{S_4}$, we have the elements 
$s_1=[2,1,3,4]$, $s_0=[0,2,3,5]$, and $[-3,5,0,8]$. The $n$ numbers in a window are pairwise noncongruent modulo $n$, and their sum is 
$n(n+1)/2$; windows are characterized by these two previous properties. 

The {\em position inversion set} of $f\in \tilde S_n$ is the set 
$$
Inv_p(f):=\{(i,j)\mid 1\leq i\leq n, i<j, f(i)>f(j)\},$$
which is also equal to 
$$
\{(i,j)\mid 1\leq i\leq n, i<j, f(i)\geq f(j)\},
$$
 since $i<j$ implies $f(i)\neq f(j)$.

The affine symmetric group is a Coxeter group, generated by the $s_i$'s, with length function $\ell(f)=\mid Inv_p(f)\mid$. The {\em left weak order}, denoted simply $\leq$, is defined on $\tilde S_n$ by its covering relation: $g$ covers $f$ if for some $i$, one has $g=s_if$ and $\ell(g)=\ell(f)+1$. 

By definition of the product in $\tilde S_n$, one has the following lemma (a variant of \cite{BB} p. 260).

\begin{lemma}\label{product}
Let $k=0,\ldots,n-1$ and $[a_1,\ldots,a_n]\in \tilde S_n$. Define $i,j$ in $\{1,2,\ldots,n\}$ by $a_i\equiv 
k+1 \mod n$ and $a_j\equiv k\mod n$. 
Then $s_k[a_1,\ldots,a_n]=[b_1,\ldots,b_n]$, with $b_j=a_j+1$, $b_i=a_i-1$, and $b_l=a_l$ if $a_l$ is noncongruent to $k$ nor to $k+1$ modulo $n$.
\end{lemma}

We also need the following result, which follows from Theorem 16 in \cite{EE} (see also \cite{BB0} Theorem 5.3).

\begin{lemma}\label{inv} Let $f,g\in \tilde S_n$ and $k=0,\ldots,n-1$. Then:

(i) the sets $Inv_p(f)$ and $Inv_p(s_kf)$ differ by one element;

(ii) $f\leq g$ if and only if $Inv_p(f)\subset Inv_p(g)$.
\end{lemma}

\subsection{A poset isomorphism}\label{--interval}

We define a mapping $F$ from the set of admitted vectors into $\mathbb Z^n$, where elements of the latter set are 
denoted by brackets $[a_1,\ldots,a_n]$. Let $v$ be an admitted vector. Then we define $F(v)=[a_1,\ldots,a_n]$, by 
\begin{equation}\label{F}
a_i=i+\sum_{p<i}v_{pi}-\sum_{i<p}v_{ip}.
\end{equation}

If we denote by $e_1,\ldots,e_n$ the canonical basis of $\mathbb Z^n$, we have equivalently 
$$
F(v)=[1,2,\ldots,n]+\sum_{r<s}v_{rs}(e_s-e_r).
$$

Now, consider the element $f_c\in \tilde S_n$ defined by its window:
$f_c=[c_1,c_2,\ldots,c_n]$ with $c_1=-(1/2)n(n-3)$ and $c_{i+1}=c_i+n-1$ for $i=1,\ldots,n-1$. Since $n-1$ is a 
generator of the additive group of integers modulo $n$, the $c_i$ are pairwise distinct modulo $n$. Moreover the sum 
of all $c_i$ is 
\begin{align*}
& (c_1+0)+(c_1+n-1)+(c_1+2(n-1))+\cdots+(c_1+(n-1)(n-1)) \\
=  &~  nc_1+\frac{(n-1)n(n-1)}{2} \\
=&  -\frac{1}{2}n^2(n-3)+\frac{1}{2}(n-1)^2n \\
= &~  \frac{1}{2}n((n-1)^2-n(n-3))\\
 = &  \frac{1}{2}n(n^2-2n+1-n^2+3n) \\
 =&\frac{n(n+1)}{2}.
\end{align*}

This shows that indeed $f_c\in \tilde S_n$. 

As an example, for $n=4$, $f_c=[-2,1,4,7]$ and for $n=5$, $f_c=[-5,-1,3,7,11]$. 

\begin{theorem}\label{interval} The mapping $F$ is a poset isomorphism from the set of admitted vectors onto the 
interval $[\id, f_c]$ with the left weak order.
\end{theorem}

Notice that this result has been proven by the second-named author for any affine Weyl group $W_a$ in \cite{C} by using a geometrical approach. We give here a combinatorial proof and further, we construct the element $f_c$.

See Figure \ref{ASGIP} for this interval. We first need several preliminary results.

%Further we deduce from a result of Nathan Reading and David E Speyer, that this lattice is semidistributive (see Corollary \ref{semi-distrib} and the remark before it).

%Notice that this result has been proven by the second-named author for any affine Weyl group $W_a$ in \cite{C}, by proving that the poset is isomorphic to an interval for the right weak order of $W_a$

\begin{lemma}\label{ai-aj} Let $F(v)=[a_1,\ldots,a_n]$ and $1\leq i<j\leq n$. Then 
$$
a_j-a_i=j-i+nv_{ij}-\sum_{i<p<j}\delta_{ipj}(v)+\sum_{p<i}\delta_{pij}(v)+\sum_{j<p}\delta_{ijp}(v).
$$
As a consequence, $a_j-a_i=nv_{ij}+r$, with $r\in\{1,\ldots,n-1\}$ and $\lfloor\frac{a_j-a_i}{n}\rfloor=v_{ij}$.
\end{lemma}

\begin{proof} We have 
\begin{align*}
a_j-a_i & =j-i+\sum_{p<j}v_{pj}-\sum_{j<p}v_{jp}-\sum_{p<i}v_{pi}+\sum_{i<p}v_{ip} \\
            & = j-i+\sum_{p<i}v_{pj}+v_{ij}+ \sum_{i<p<j}v_{pj}-\sum_{j<p}v_{jp}-\sum_{p<i}v_{pi}+\sum_{i<p<j}v_{ip}+v_{ij}+\sum_{j<p}v_{ip} \\
            & =j-i+2v_{ij}+\sum_{i<p<j}(v_{ip}+v_{pj})+\sum_{p<i}(v_{pj}-v_{pi})+\sum_{j<p}(v_{ip}-v_{jp}) \\
            & =j-i+2v_{ij}+\sum_{i<p<j} (v_{ij}-\delta_{ipj}(v))+\sum_{p<i}(v_{ij}+\delta_{pij}(v))+\sum_{j<p}(v_{ij}+\delta_{ijp}(v)) \\
            & =j-i+nv_{ij}-\sum_{i<p<j}\delta_{ipj}(v)+\sum_{p<i}\delta_{pij}(v)+\sum_{j<p}\delta_{ijp}(v).
\end{align*}

Let $r=j-i-\sum\limits_{i<p<j}\delta_{ipj}(v)+\sum\limits_{p<i}\delta_{pij}(v)+\sum\limits_{j<p}\delta_{ijp}(v)$. Then $a_j-a_i=nv_{ij}+r.$ Moreover, since the $\delta$'s take the values 0 and 1, $r$ is in the interval $[j-i-(j-i-1),j-i+i-1+n-j]=[1,n-1]$.
\end{proof}

\begin{proposition}\label{ImF} The image of $F$ is contained in the set of windows. If $F(v)=[a_1,\ldots,a_n]$, then $a_1<\cdots<a_n$.
\end{proposition}

\begin{proof} For the first assertion, it is enough to show that the $a_i$ as defined above are pairwise noncongruent modulo $n$, and that their sum is $n(n+1)/2$.

Let $i<j$. Because of Lemma \ref{ai-aj} we already know that $a_i-a_j$ is nonzero modulo $n$. Further we have 
\begin{align*}
\sum\limits_{i=1}^na_i  & =\sum\limits_{i=1}^n(i+\sum_{p<i}v_{pi}-\sum_{i<p}v_{ip}) 
			    = \sum\limits_{i=1}^ni+\sum_{p<i}v_{pi}-\sum_{i<p}v_{ip} 
			    = \sum\limits_{i=1}^ni 
			    = \frac{n(n+1)}{2}.
\end{align*}

For the last assertion, let $1\leq i<j\leq n$. By Lemma \ref{ai-aj}, $a_j-a_i>0$.
\end{proof}

\vspace*{\stretch{1}}
\begin{figure}[h!]
\begin{center}
\includegraphics[scale=0.65]{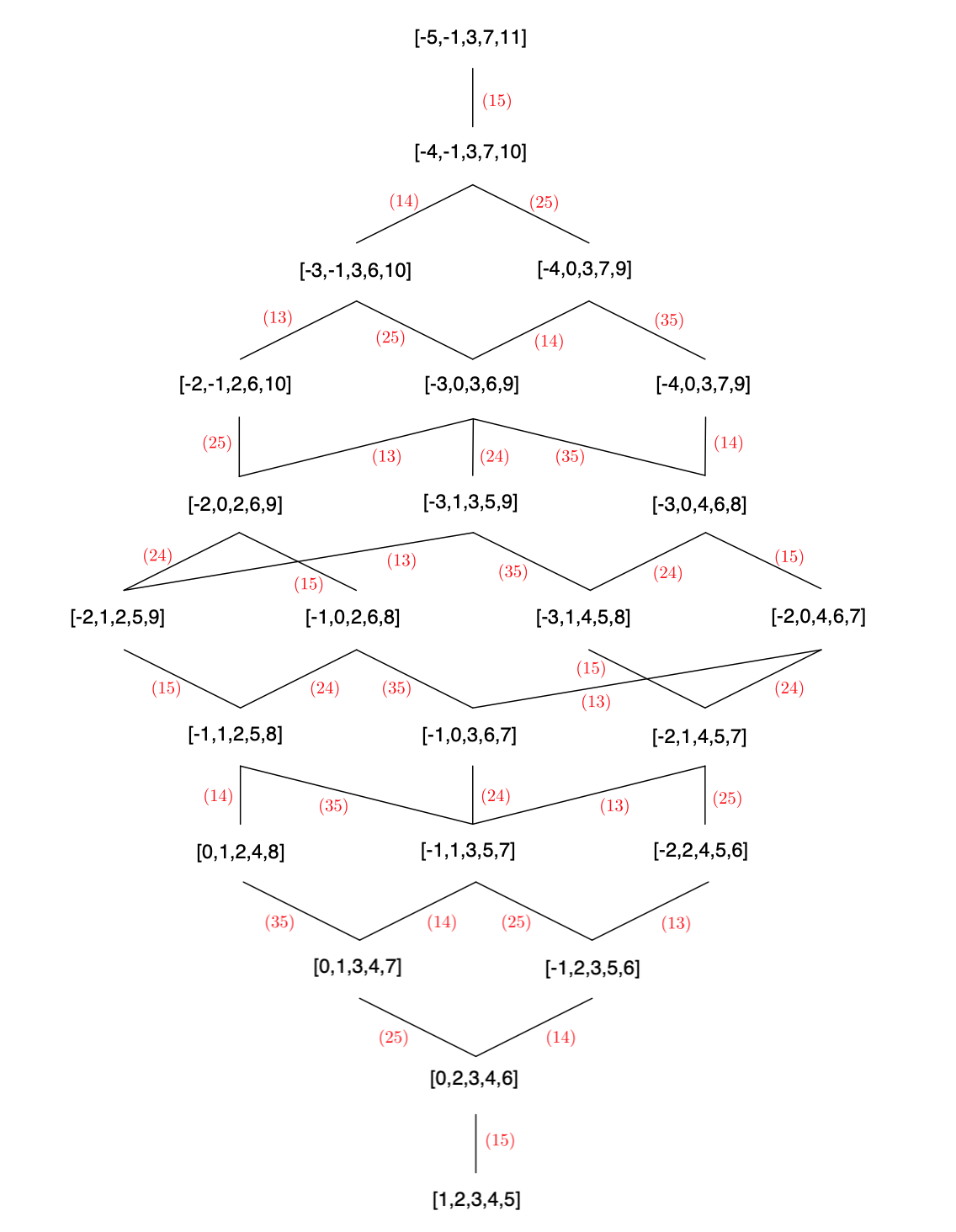}
\end{center}
\caption{The interval in the affine symmetric group $\tilde S_5$}\label{ASGIP}
\end{figure}
\vspace*{\stretch{1}}

\newpage

\begin{lemma}\label{Inv-ordre} Suppose that $f,g\in\tilde S_n$ with $f=[a_1,a_2,\ldots,a_n]$, $g=[b_1,b_2,\ldots,b_n]$, $a_1\leq\cdots \leq a_n$ and $b_1\leq\cdots\leq b_n$. Then:

(i) $Inv_p(f)$ is the set of pairs $(i,p+kn)$ with $1\leq p<i\leq n$ and $k=1,\ldots,\lfloor\frac{a_i-a_p}{n}\rfloor$.

(ii) $f\leq g$ if and only if for any $i<j$, $\lfloor\frac{a_j-a_i}{n}\rfloor\leq \lfloor\frac{b_j-b_i}{n}\rfloor$.
\end{lemma}

Note that assertion (ii) is a particular case of  Theorem 16 (3) in \cite{EE}.

\begin{proof} (i) A pair as in the statement is an inversion, since $i<p+kn$ and $f(i)=a_i\geq f(p+kn)=f(p)+kn=a_p+kn$ because 
$k\leq  \lfloor\frac{a_i-a_p}{n}\rfloor\leq\frac{a_i-a_p}{n}$. Conversely, let $(i,j)$ be an inversion of $f$. Then $1\leq i\leq n$. Since $j > i \geq 1$, we may write $j=p+kn,1\leq p\leq 
n,k\geq 0$. We have $f(i)>f(j)=f(p)+kn$. We cannot have $p \geq i$, since otherwise we would have, by the increasing 
property, $f(i) \leq f(p) \leq f(j)$, a contradiction. Hence $p < i$ and since $p+kn=j>  i$, we must have $k \geq 1$. 
Moreover, $kn < f(i)-f(p)=a_i-a_p$, hence $k \leq\lfloor \frac{a_i-a_p}{n} \rfloor$.

(ii) follows from (i).
\end{proof}

\begin{lemma}\label{cover2} Consider two admitted vectors $v,v'$ and $1\leq i<j\leq n$ such that $v'_{ij}=v_{ij}+1$ 
whereas $v'_{rs}=v_{rs}$ for the other 
coordinates. Let $k=a_j\mod n$, where $a_j$ is the $j$-th coordinate of $F(v)$. Then in $\tilde S_n$, one has the equality $F(v')=s_kF(v)$ and $\ell(F(v'))=\ell(F(v))
+1$.
\end{lemma}

\begin{proof} Let $F(v)=[a_1,\ldots,a_n]$ and $F(v')=[a'_1,\ldots,a'_n]$. Then by Eq.(\ref{F}):

$$
a'_i=i+\sum\limits_{p<i}v'_{pi}-\sum\limits_{i<p}v'_{ip}=i+\sum\limits_{p<i}v_{pi}-\sum\limits_{i<p}v_{ip}-1=a_i-1,
$$
 and similarly $a'_j=a_j+1$, $a'_l=a_l$ if $l\neq i,j$. 

By Lemma \ref{ai-aj} and Lemma \ref{monter}, $a_j-a_i$ is equal to $j-i+(i-1)+(n-j)=n-1\equiv -1 \mod n$. Hence $a_i\equiv k+1\mod n$.

Let $s_kF(v)=[b_1,\ldots,b_n]$. By Lemma \ref{product}, we deduce that $b_i=a_i-1$ and $b_j=a_j+1$, whereas $b_l=a_l$ for the other $l$. Hence $b_l=a'_l$ for all $l$, that is, $F(v')=s_kF(v)$.

We have $a_j=a_i+n-1$. Hence 
$$
F(v)=[\ldots,a_i,\ldots,a_j,\ldots]=[\ldots,a_i,\ldots,a_i+n-1,\ldots],
$$ 
and
\begin{align*}
F(v') & =[\ldots,a'_i,\ldots,a'_j,\ldots]=[\ldots,a_i-1,\ldots,a_j+1,\ldots] \\
        & =[\ldots,a_i-1,\ldots,a_i+n,\ldots].
\end{align*}
 Thus $F(v')$ has the position inversion $(j,i+n)$ (because $F(v')$ send this pair 
onto $(a_i+n,a_i+n-1)$), but $F(v)$ has not this inversion (because the pair $(j,i+n)$ is sent by $F(v)$ onto $
(a_i+n-1,a_i+n)$). It follows from Lemma \ref{inv} (i) that the length of $F(v')$ is one more than the length of $F(v)$.
\end{proof}

\begin{lemma}\label{FG=id} Let $[a_1,\ldots,a_n]\in\tilde S_n$. Then
$$
a_i=i-\sum_{i<p}\lfloor\frac{a_p-a_i}{n}\rfloor+
\sum_{p<i}\lfloor\frac{a_i-a_p}{n}\rfloor.
$$
\end{lemma}

This formula could be proved as a corollary of \cite{Sh} Theorem 4.1, but we give a direct proof.

\begin{proof} By definition of $\tilde S_n$, there exists a permutation $\sigma\in S_n$ and integers $b_1,\ldots,b_n$ 
such that $a_i=\sigma(i)+nb_i$ and $b_1+\cdots+b_n=0$. One has 
\begin{align*}
A & := i-\sum_{i<p}\lfloor\frac{a_p-a_i}{n}\rfloor+\sum_{p<i}\lfloor\frac{a_i-a_p}{n}\rfloor \\
   & = i-\sum_{i<p}\lfloor\frac{\sigma(p)+nb_p-\sigma(i)-nb_i}{n}\rfloor+
\sum_{p<i}\lfloor\frac{\sigma(i)+nb_i-\sigma(p)-nb_p}{n}\rfloor.
\end{align*}

Note that for integers $b,s$, $\lfloor \frac{nb+s}{n} \rfloor=b+\lfloor \frac{s}{n} \rfloor$, that $-1<\frac{\sigma(j)-\sigma(k)}{n}<1$, and that $\sigma(j)-\sigma(k)$ has integer part $0$ or $-1$ depending on $\sigma(j)-\sigma(k)\geq 0$ or $<0$. Thus
$$A=i-\sum_{i<p}(b_p-b_i)+\sum_{p<i}(b_i-b_p)-\sum_{i<p,\sigma(p)<\sigma(i)}(-1)+\sum_{p<i,\sigma(i)<\sigma(p)}(-1).
$$
We claim that $\sum\limits_{i<p,\sigma(p)<\sigma(i)}1=\sigma(i)-i+\sum\limits_{p<i,\sigma(i)<\sigma(p)}1$. It follows that
$$
A=i+(n-1)b_i-\sum_{p\neq i}b_p+\sigma(i)-i=\sigma(i)+nb_i=a_i.
$$
because $-\sum\limits_{p\neq i}b_p=b_i$. 

In order to prove the claim, note that since $\sigma$ is a permutation, one has for fixed $i$: 
$$
\sum\limits_{i<p,\sigma(p)<\sigma(i)}1+\sum\limits_{i>p,\sigma(p)<\sigma(i)}1=\sum\limits_{\sigma(p)<\sigma(i)}1=\sigma(i)-1,$$ 
and
 $$\sum\limits_{i>p,\sigma(p)<\sigma(i)}1+\sum\limits_{i>p,\sigma(p)>\sigma(i)}1=\sum\limits_{i>p}1=i-1.
 $$
  Hence 
$$
\sum\limits_{i<p,\sigma(p)<\sigma(i)}1=\sigma(i)-1-\sum\limits_{i>p,\sigma(p)<\sigma(i)}1=\sigma(i)-1-(i-1-\sum\limits_{i>p,
\sigma(p)>\sigma(i)}1).
$$
\end{proof}

We prove now Theorem \ref{interval}. The mapping $G=F^{-1}$, defined in the proof, appears implicitly in \cite{EE} p.20 and Eq. (4.2.1) p. 375 in \cite{S}.

\begin{proof}[Proof of Theorem \ref{interval}]
1. It is easy to see that $F$ sends the null admitted vector onto $\id=[1,2,\ldots,n]$. We show that $F(w)=f_c$, where $w$ is the largest admitted vector, so that $w_{ij}=j-i-1$ (Corollary \ref{largest}). Let 
$F(w)=[a_1,\ldots,a_n]$. Then 
\begin{align*}
a_1=1-\sum\limits_{1<p\leq n}w_{1p} & = 1-\sum\limits_{1<p\leq n}(p-2)  = 1- \frac{(n-1)(n-2)}{2} \\
													  & = \frac{2-n^2+3n-2}{2} \\
													  & =- \frac{n(n-3)}{2}.
\end{align*}

Moreover, by Lemma \ref{ai-aj}, noting that 
$$
a_{i+1}-a_i=1+\sum\limits_{p<i}\delta_{p,i,i+1}(w)+\sum\limits_{i+1<p}\delta_{i,i+1,p}(w)~\text{~~and~~}~v_{i,i+1}=0,
$$
we have by Corollary \ref{largest} that it is equal to $1+(i-1)+(n-i-1)=n-1$. It follows that $a_i=c_i$, and $F(w)=f_c$.

2. Let $v$ be an admitted vector and $F(v)=f$. We have $v_{ij}\leq w_{ij}$, so that by Proposition \ref{ImF}, Lemma \ref{Inv-ordre} (ii) and Lemma \ref{ai-aj}, one 
has $f\leq f_c$. By the same results, $\id\leq f$. Thus $f\in[\id,f_c]$.

3. Let  $0\leq i<j\leq n$. The injectivity of $F$ follows from the equality $v_{ij}=\lfloor\frac{a_j-a_i}{n}\rfloor$, see Lemma \ref{ai-aj}.

4. Given $f\in[\id,f_c]$, define $G(f)=v\in \mathbb Z^T$ by $v_{ij}=\lfloor\frac{a_j-a_i}{n}\rfloor$.
The previous paragraph shows that $G\circ F=\id$. 

We show that the image of $G$ is contained in the set of admitted vectors. For this note that for an integer $x$, $\lfloor x/n\rfloor$ is the 
quotient of the Euclidean division of $x$ by $n$; therefore $\lfloor (x+y)/n\rfloor=\lfloor x/n\rfloor+\lfloor y/n\rfloor+\epsilon$, with $
\epsilon=0$ or $1$. It follows that if $i<j<k$, $\lfloor \frac{a_k-a_i}{n}\rfloor-\lfloor \frac{a_k-a_j}{n}\rfloor-\lfloor \frac{a_j-a_i}{n}
\rfloor$ is equal to 0 or 1, and $v_{ij}$ is thus an admitted vector.

Observe that $f =[a_1,\ldots,a_n]\in\tilde S_n$ has no inversion of the form $(i,j)$ with $1\leq i<j\leq 
n$ if and only if $a_1<\ldots <a_n$. This is the case in particular for $f=f_c$. Therefore each $f\in[\id,f_c]$ satisfies this increasing condition, because $f\leq f_c$ implies that by Lemma \ref{inv}, each inversion of $f$ is an inversion of $f_c$. The
nonnegativity of the coefficients of $G(f)$ follows.

5. The equality $F\circ G=\id$ follows from Lemma \ref{FG=id}.

6. It remains to show that $F$ and $G$ are increasing functions. That $F$ is increasing follows from Lemma \ref{cover2} and the description in Lemma \ref{cover-admitted} of the covering relation for admitted vectors.

Suppose now that $f,g\in [\id,f_c]$ and $f\leq g$. Then $f,g$ satisfy the hypothesis of Lemma \ref{Inv-ordre} and it 
follows by (ii) in this lemma that $G(f)\leq G(g)$.
\end{proof}

By composing the bijections, we obtain a direct poset isomorphism from the interval $[\id,f_c]$ into the set of circular permutations. Recall that $\tilde S_n$ is the semidirect product of $S_n$ and of the additive group $\{(k_1,\ldots,k_n)\in\mathbb Z^n, \sum_ik_i=0\}$. The canonical projection $\tilde S_n\to S_n$ is defined by $f\mapsto \sigma$, with $\sigma(i)=f(i) \MOD n$, where $x \MOD n$ is the unique representative modulo $n$ of $x$ in $\{1,\ldots,n\}$.

\begin{corollary} The poset isomorphism $V^{-1}\circ F^{-1}$ from  $[\id,f_c]$ into the set of circular permutations in $S_n$ is defined by $f\mapsto (\sigma)$, where $\sigma$ is as above.
\end{corollary}

\begin{proof} This follows from the definition of $G$ and of the proof of Proposition \ref{path} at the end of Section \ref{vectors}.
\end{proof}

In a Coxeter group, an interval $I$ is always a lattice. Indeed it is well known that the right weak order is a complete meet-semilattice (see \cite{BB} Section 3.2). Therefore each pair of elements $x,y \in I$ has an infimum. Since $I$ admits a upper bound it turns out that each pair of elements $x,y \in I$ has also a supremum.  Further, it is known that such an interval is a semidistributive lattice (see \cite{R} Theorem 8.1 and the comment after it). Consequently, the poset considered in the present article has the following property.

\begin{corollary}\label{semi-distrib} The poset is a semidistibutive lattice.
\end{corollary}

%$\cite[Ch. VI, $\S$ 1.10]{BOURB}

We leave to the reader to verify that $f_c$ is an involution in $\tilde S_n$ if $n$ is odd, but not when $n$ is even.

Define the elements $s_{ijk}$ of $\tilde S_n$, where $1\leq i<j\leq n$ and $k\in \mathbb Z$ for any $x\in\{1,2,\ldots,n\}$ by
$$s_{ijk}(x)=\left\{
\begin{array}{cccc}
j-kn& \mbox{if} &x=i \\
i+kn& \mbox{if} &x=j \\
x & \mbox{if} &x\neq i,j.
\end{array}
\right.
$$

\begin{proposition} Let $v$ be an admitted vector and $f=F(v)$. Consider the upwards path, in the Hasse diagram of 
the poset of admitted vectors, from the null vector to $v$, with successive labels $(i_1,j_1),\ldots,(i_r,j_r)$, where $r$ is the rank of $v$ in the 
poset. For $p=1,\ldots,r$, let $k_p$ be the number of $q\leq p$ such that $(i_q,j_q)=(i_p,j_p)$. Then in $\tilde S_n$, 
one has the factorization
$$ f=s_{i_1j_1k_1}\circ\cdots \circ s_{i_rj_rk_r}.$$
\end{proposition}

The proof is left to the reader.
%\begin{proof}
%\end{proof}

Note that it follows from Theorem \ref{interval}, and from Corollary 3.2.8 in \cite{BB}, that the M\"obius function of the poset takes its values in $\{-1,0,1\}$ \label{mobius}.

\section{lines diagrams}\label{lines}

%We might define a nice representation of the bijection $V_0$ using lines diagrams. We denote by $D$ the unit circle and by $C=D\times [0,1]$. We put the numbers from 1 to $n$ using the lexicographical order clockwise on $D\times\{1\}$ and the numbers from 1 to $n$ using any order on $D\times\{0\}$.
%We define a braid to be the image of a map from $[0,1]$ to $C$ that starts on $1\leq i\leq n$ on $D\times\{1\}$ and finish on $i$ on $D\times\{0\}$.

We define a \emph{lines diagram} $D$ as follows: we take a cylinder, and we put in clockwise order the numbers from 1 to $n$ at the top and in any order at the bottom. For any $1\leq i \leq n$, we trace a continuous oriented line that goes from $i$ at the top to $i$ at the bottom and doesn't cross itself.
We can represent this diagram on a 2-dimensional surface by either projecting it like in Figure \ref{LD1} or unrolling it like in Figure \ref{LD2}.

By turning clockwise around the cylinder, we can list the numbers at the bottom starting with $1$ and we obtain a word $w=1 i_1\cdots i_{n-1}$. We then say that a lines diagram $D$ represents the circular permutation $(w)$. For example, in Figures \ref{LD1} and \ref{LD2} below, we draw two lines diagrams representing the same permutation $(1,6,4,2,3,5)$.

\begin{figure}[h!]
\includegraphics[scale=0.65]{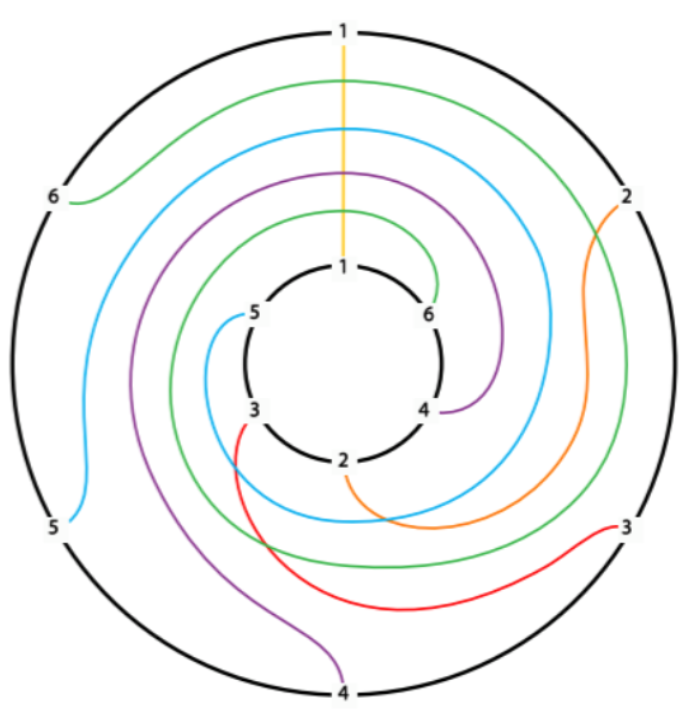} 
\caption{Circular lines diagram representing the circular permutation $(1,6 ,4 ,2, 3, 5)$}\label{LD1}
\end{figure}

\begin{figure}[h!]
\includegraphics[scale=0.7]{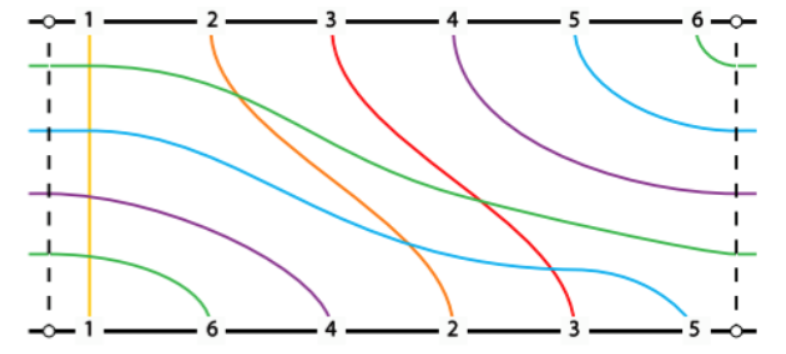} 
\caption{Cylindrical lines diagram representing the circular permutation $(1,6 ,4 ,2, 3, 5)$}\label{LD2}
\end{figure}

We denote by $L_i(D)$ the line of $D$ that goes from $i$ at the top to $i$ at the bottom. We define $L_{i,j}(D)$ to be number of times that $L_i(D)$ and $L_j(D)$ cross each other.

For $i<j$, a crossing between the lines $L_i(D)$ and $L_j(D)$ is called a \emph{legal crossing} if 
\begin{itemize}
	\item[i)] $i+1<j$,
	\item[ii)] $L_j(D)$ crosses $L_i(D)$ from the right side of $L_i(D)$ to the left side of $L_i(D)$ when one follows $L_i(D)$ from the top to the bottom of the cylinder. See figure \ref{LD3}.
\end{itemize}

We say that a lines diagram is \emph{legal} if all its crossings are legal. Note that a legal lines diagram must satisfy $L_{i,i+1}(D)=0$ for all $1\leq i<n$. The diagrams in Figures \ref{LD1} and \ref{LD2} are examples of legal diagrams.

\begin{figure}[h!]
\includegraphics[scale=0.51]{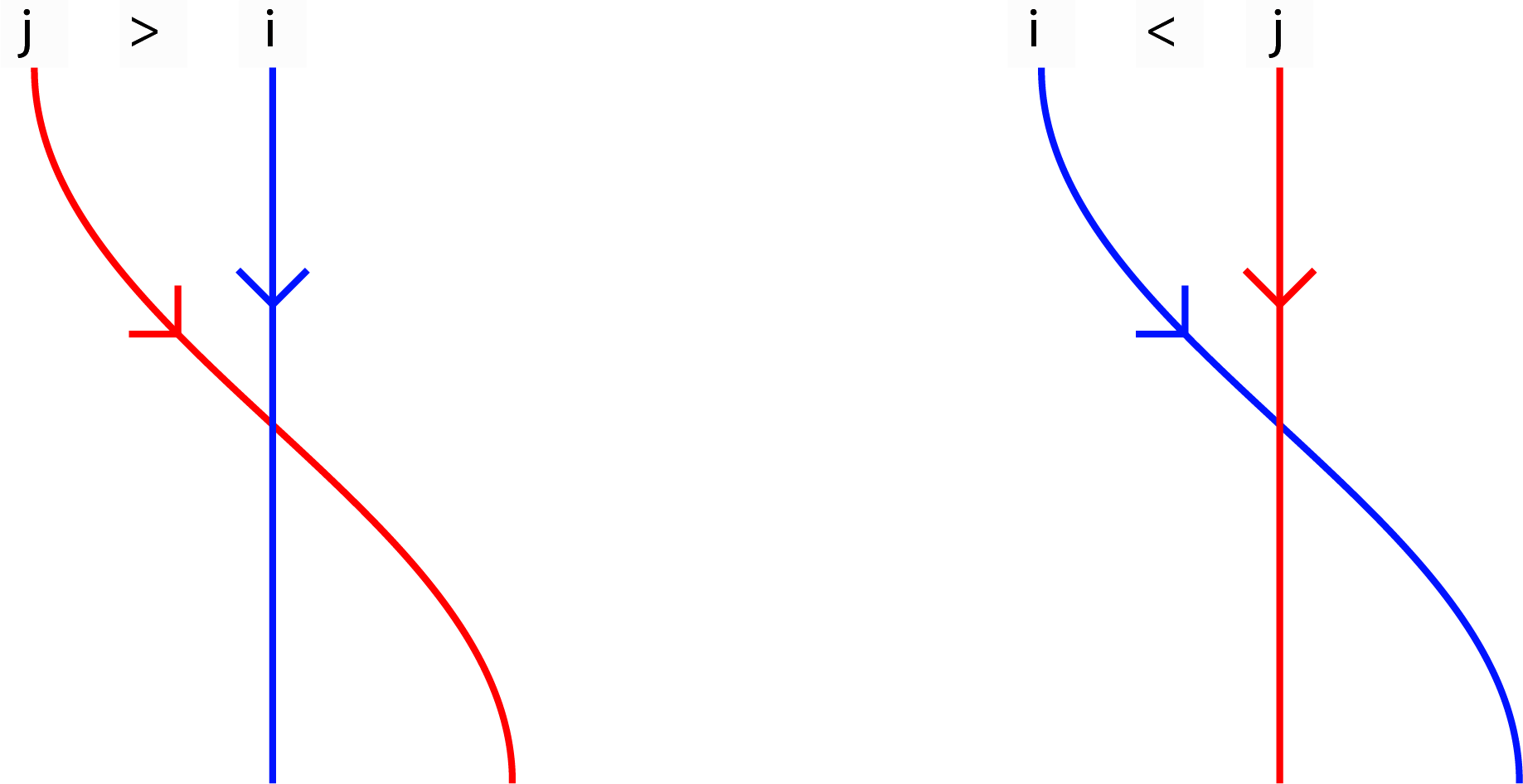} 
\caption{The left crossing is legal while the right one is illegal when $i<j$}\label{LD3}
\end{figure}

We call the \emph{trivial lines diagram of $n$}, the legal lines diagram with no crossing. The trivial lines diagram of $n$ represents the permutation $(1,2,...,n)$.

Consider a circular permutation $\sigma$; we build a legal lines diagram as follows: we take a strictly increasing path $P=((1,2,...,n)=\sigma_0,\sigma_1,...,\sigma_h=\sigma)$ in the poset of circular permutations that goes from the bottom of the poset to the permutation $\sigma$. Denote by $D_0$ the trivial lines diagram representing $\sigma_0$. Suppose we have build the legal lines diagrams $D_{k-1}$ representing $\sigma_{k-1}$, $k\geq 1$. We now build $D_k$. Denote by $(i,j)$ the large circular descent labelling the step from $\sigma_{k-1}$ to $\sigma_k$ in the poset. We build the diagram $D_k$ by interchanging the number $j$ and $i$ at the bottom of the cylinder and, therefore, forcing $L_j(D_k)$ to cross $L_i(D_k)$ right before reaching the bottom of the cylinder. We know that $(i,j)$ is a large circular descent of $\sigma_{k-1}$ and therefore the crossing that is created is legal. 

Remembering that, for a circular permutation $\sigma$, each ascending step in the poset is equivalent to exchange $i$ and $j$ for a large circular descent $(i,j)$, it is straightforward that the lines diagram $D_{\sigma}$ obtained with the algorithm above has the property that $L_{i,j}(D)=v_{i,j}$, where $v_{i,j}$ is a coordinate of the vector $v_{\sigma}$. We will later show that any legal lines diagram representing $\sigma$ has this property.

Considering two lines $L_i(D)$ and $L_j(D)$ of a lines diagram $D$ that are crossing each other, we define the \emph{triangle $i-j$ of $D$}, denoted $T_{i,j}(D)$, to be the triangle whose vertices are: the last intersection of $L_i(D)$ and $L_j(D)$ and the points $i$ and $j$ on the bottom of the cylinder.

\begin{lemma}\label{triangle}
	Let $D$ a legal lines diagram. If two lines are crossing each other, $D$ has a triangle that does not intersect any other line than its boundary.
\end{lemma}
\begin{proof}
	Call \emph{regular} triangle a triangle such that at least one of its non-bottom edges does not cross any line. There always exist such triangles. Indeed, follow a line $L_i(D)$ and consider the triangle whose top vertex is the last crossing of $L_i(D)$.
	Order regular triangles by inclusion of their surfaces.
	Let $T$ be a minimal element of this poset. Then, $T$ satisfies the conclusion of this lemma.
%	
%	If $L_i(D)$ and $L_j(D)$ intersect, denote by $\psi_{i,j}$ their last intersection of $L_i(D)$ and $L_j(D)$.
%	Suppose, for $i<j$, that $L_i(D)$ and $L_j(D)$ are crossing each other in $\psi_{i,j}$.
%	We can take $L_i(D)$ and $L_j(D)$ such that $L_i(D)$ is the last line that crosses $L_j(D)$ by taking the largest integer $j\leq n$ such that $L_j(D)$ crosses lines in $D$. Therefore, if a line $L_k(D)$ is crossing $L_i(D)$, it is the last crossing between the two along $L_i{D}$.
%	
%	If $\psi_{i,j}$ is the last crossing along $L_i(D)$, $T_{i,j}(D)$ is a triangle that does not contain any segment of line.
%	
%	Otherwise we can consider $L_{j_1}(D)$ the last line that crosses $L_i(D)$. Note that $T_{i,j_1}(D)$ is contained in $T_{i,j}(D)$. We then consider $L_{i_1}(D)$ the last line that crosses $L_{j_1}(D)$. We also have that $T_{j_1,i_1}(D)$ is contained in $T_{i,j_1}(D)$.
%	We are building a decreasing sequence of triangles for the inclusion order. Therefore, this process will stop and we will end up with a triangle $T_{i_k,j_k}(D)$ that does not contain any segment of line.
\end{proof}

\begin{theorem}\label{untie}
	Let $D$ be a non-trivial legal lines diagram representing a circular permutation $\sigma$. There exists a large circular ascent $a_h$ in $\sigma$ such that we can create a legal diagram $D'$ that represent $(a_h a_{h+1})\sigma(a_h a_{h+1})$ where $L_{k,l}(D')=L_{k,l}(D)$ for any pair $k,l$ except for $a_h, a_{h+1}$ where $L_{a_h, a_{h+1}}(D')=L_{a_h, a_{h+1}}(D)-1$.
\end{theorem}
\begin{proof}
	By lemma \ref{triangle} we know that there exists in $D$ a pair $i,j$ such $T_{i,j}(D)$ is not crossed by any line. Therefore, $i$ is a large circular ascent of $\sigma$, with $j$ following $i$. By interchanging $i$ and $j$ on the bottom of the cylinder, we can untie the last crossing between $L_i(D)$ and $L_j(D)$ and keep the rest of the diagram unchanged. This proves the result.
\end{proof}

\begin{corollary}\label{cor coeff}
	Let $\sigma$ be a circular permutation and $D$ a legal lines diagram representing $\sigma$. Considering $v_{\sigma}$ the vector associated to $\sigma$, $D$ has the property that $L_{i,j}(D)=v_{i,j}$ for any pair $i,j$.
\end{corollary}
\begin{proof}
	It is easy to convince ourself that the only legal lines diagram that represents $(1,2,...,n)$ is the trivial lines diagram, which has no crossing.
	
	The rest follows from theorem \ref{untie} and the poset of circular permutations.
\end{proof}

In Figures \ref{LD1} and  \ref{LD2} above, we drew two lines diagrams for the same permutation; it illustrates the bijection $V$ from circular permutations towards admitted vectors. Figure \ref{LD1} uses two concentric circles, while Figure \ref{LD2} is drawn on a cylinder, represented by a rectangle where the two vertical sides are identified. 
Thus, because of Corollary \ref{cor coeff} we know the admitted vector $\lambda_{\sigma}$ associated to the circular permutation $\sigma$, see Figure \ref{adm vector sigma}.

\begin{figure}[h!]
\includegraphics[scale=0.5]{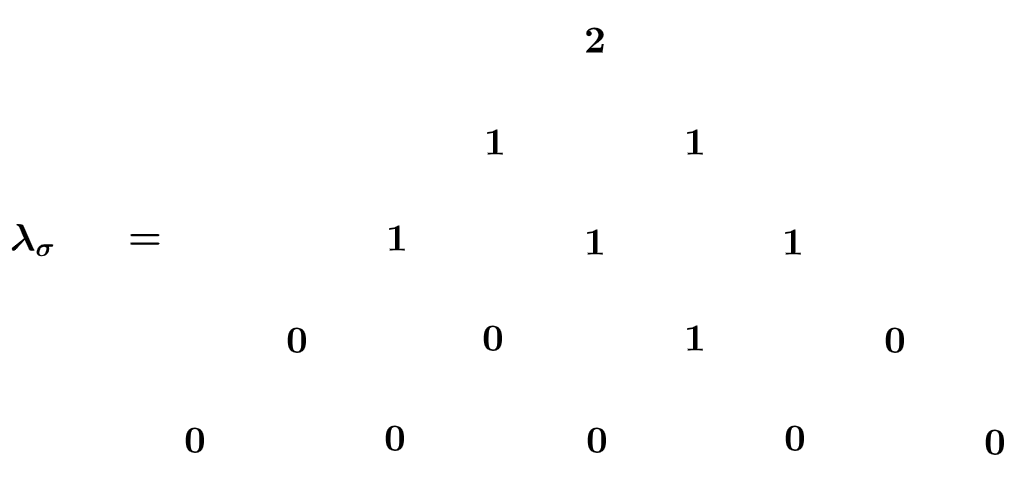} 
\caption{Admitted vector associated to the circular permutation $(1,6 ,4 ,2, 3, 5)$}
\label{adm vector sigma}
\end{figure}

{\em Acknowledgments}: The authors thank Nicolas England, Christophe Hohlweg, Hugh Thomas and Richard Stanley for useful discussions 
and mails. 

\nocite{*}
\bibliographystyle{plain}
\bibliography{ordering_circular_permutations.bib}

\begin{thebibliography}{10}

\bibitem{BB0}
Anders Bj\"{o}rner and Francesco Brenti.
\newblock Affine permutations of type a.
\newblock {\em Electronic Journal of Combinatoics 3}, (R18):35 p, 1996.

\bibitem{BB}
Anders Bj\"{o}rner and Francesco Brenti.
\newblock {\em Combinatorics of {C}oxeter groups}, volume 231 of {\em Graduate
  Texts in Mathematics}.
\newblock Springer, New York, 2005.

\bibitem{C}
Nathan Chapelier-Laget.
\newblock Shi variety corresponding to an affine weyl group.
\newblock {\em arXiv preprint: 2010.04310}, 2020.

\bibitem{C2}
Nathan Chapelier-Laget.
\newblock A symmetric group action on the irreducible components of the {S}hi
  variety associated to ${W}(\widetilde{A}_n)$.
\newblock {\em arXiv preprint: 2010.05602}, 2020.

\bibitem{EE}
Henrik Erikson and Kimmo Erikson.
\newblock Affine weyl groups as infinite permutations.
\newblock {\em Electronic Journal of Combinatoics 5}, (R18):32 p, 1998.

\bibitem{P}
T.~Kyle Petersen.
\newblock Eulerian numbers.
\newblock {\em Birkha\"user/Springer}, 2015.

\bibitem{R}
Nathan Reading and David~E Speyer.
\newblock Sortable elements in infinite coxeter groups.
\newblock {\em Transactions of the American Mathematical Society},
  363:699--761, 2011.

\bibitem{Sa}
Bruce Sagan.
\newblock The symmetric group.
\newblock {\em Springer}, 2000.

\bibitem{Sh}
Jian~Yi Shi.
\newblock On two presentations of the affine weyl groups of classical types.
\newblock {\em Journal of Algebra}, 221:360--382, 1999.

\bibitem{S}
Richard~P Stanley.
\newblock Enumerative combinatorics.
\newblock {\em Wadsworth and Brooks/Cole}, 1, 1986.

\end{thebibliography}

\end{document}